\documentclass[reqno]{amsart}
\usepackage[backref=page]{hyperref}
\usepackage{cite}
\usepackage{enumitem}

\theoremstyle{plain}
\newtheorem{theorem}{Theorem}
\newtheorem{proposition}[theorem]{Proposition}

\newtheorem{lemma}[theorem]{Lemma}

\theoremstyle{definition}

\theoremstyle{remark}
\newtheorem{remark}{Remark}

\allowdisplaybreaks

\begin{document}
	
	\author{Phuong Le}
	\address{Phuong Le$^{1,2}$ (ORCID: 0000-0003-4724-7118)\newline
		$^1$Faculty of Economic Mathematics, University of Economics and Law, Ho Chi Minh City, Vietnam; \newline
		$^2$Vietnam National University, Ho Chi Minh City, Vietnam}
	\email{phuongl@uel.edu.vn}
	
	
	
	\title[Monotonicity and 1D symmetry for $p$-Laplace problems]{Quasilinear elliptic problems with singular nonlinearities in half-spaces}
	\begin{abstract}
		We study the monotonicity and one-dimensional symmetry of positive solutions to the problem $-\Delta_p u = f(u)$ in $\mathbb{R}^N_+$ under zero Dirichlet boundary condition, where $p>1$ and $f:(0,+\infty)\to\mathbb{R}$ is a locally Lipschitz continuous function with a possible singularity at zero. Classification results for the case $f(u)=\frac{1}{u^\gamma}$ with $\gamma>0$ are also provided.
	\end{abstract}
	
	\maketitle
	
	\section{Introduction}
	This paper is concerned with the qualitative properties of solutions to the $p$-Laplacian problem
	\begin{equation}\label{main}
		\begin{cases}
			-\Delta_p u = f(u) & \text{ in } \mathbb{R}^N_+,\\
			u>0 & \text{ in } \mathbb{R}^N_+,\\
			u=0 & \text{ on } \partial\mathbb{R}^N_+,
		\end{cases}
	\end{equation}
	where $p>1$ and $f:(0,+\infty)\to\mathbb{R}$ is a locally Lipschitz continuous function with a possible singularity at zero. As usual, $\Delta_p$ denotes the $p$-Laplacian and the upper half-space $\mathbb{R}^N_+$ is defined as
	\[
	\mathbb{R}^N_+ := \{x:=(x',x_N)\in\mathbb{R}^N \mid x_N>0 \}.
	\]
	When $f(t)$ blows up as $t\to0^+$, this problem is usually called a singular quasilinear elliptic problem. One may think of a prototype for \eqref{main} as
	\begin{equation}\label{singular}
		\begin{cases}
			-\Delta_p u = \dfrac{1}{u^\gamma} + g(u) & \text{ in } \mathbb{R}^N_+,\\
			u>0 & \text{ in } \mathbb{R}^N_+,\\
			u=0 & \text{ on } \partial\mathbb{R}^N_+,
		\end{cases}
	\end{equation}
	where $\gamma>1$ and $g:[0,+\infty)\to\mathbb{R}$ is a locally Lipschitz continuous function.
	The problems in half-spaces like this one are important because the half-space represents the simplest unbounded domain with an unbounded boundary. When performing a blow-up near the boundary in smooth domains, the problems often reduce to those in a half-space. This reduction is essential for understanding the behavior of solutions near boundaries in more complex domains, see \cite{MR3912757,MR4044739}.
	
	The monotonicity of solutions in the $x_N$-direction to problem \eqref{main} when $f$ is not singular was studied by several authors in the literature via the moving plane method. Berestycki, Caffarelli, and Nirenberg demonstrated in \cite{MR1395408,MR1655510} that if \( f:[0,+\infty) \) is a Lipschitz continuous function with \( f(0) \geq 0 \), then any positive classical solution of \eqref{main} with \( p = 2 \) is increasing in the \( x_N \)-direction, and additionally, \( \frac{\partial u}{\partial x_N} > 0 \) in \( \mathbb{R}^N_+ \). Earlier monotonicity results for the problem \eqref{main} with \( p = 2 \) can also be found in the works of Dancer \cite{MR1190345,MR2525168}. When \( f \) is only locally Lipschitz continuous on $[0,+\infty)$, similar monotonicity can be established for positive solutions that are bounded on strips, as shown in \cite{MR4142367,MR2254600}. The case where \( f(0) < 0 \) is more complex, with a complete proof of monotonicity for solutions in this case being available only in dimension \( N = 2 \) in the works of Farina and Sciunzi \cite{MR3593525,MR3641643}.
	
	Studying problem \eqref{main} when \( p \neq 2 \) presents several challenges, primarily due to the nonlinearity of the \( p \)-Laplacian for \( p \neq 2 \). This nonlinearity indicates that comparison principles are not equivalent to maximum principles for the \( p \)-Laplacian. Additionally, the operator's singular or degenerate nature (corresponding to \( 1 < p < 2 \) and \( p > 2 \), respectively) leads to a lack of \( C^2 \) regularity in the solutions at their critical points. Overcoming these difficulties, the moving plane method for problem \eqref{main} has been extensively developed in a series of papers by Farina, Montoro, Sciunzi, and their collaborators \cite{MR3303939,MR2886112,MR4439897,MR3752525,MR3118616}. Assuming that \( f \) is positive on \((0,+\infty)\), they proved the monotonicity of solutions in the \( x_N \)-direction for the case \( 1 < p < 2 \) in \cite{MR3303939,MR2886112}. For \( p > 2 \), the monotonicity result was established in \cite{MR3752525,MR3118616} under the conditions that \( f \) is either sublinear or superlinear. Recently, it was shown in \cite{MR4439897} that when \( \frac{2N+2}{N+2}<p<2 \), the requirement for \( f \) to be positive can be relaxed.
	In all of these works, it is assumed that \( f:[0,+\infty) \) is a locally Lipschitz continuous function and solutions $u$ to \eqref{main} satisfy $u \in C^{1,\alpha}_{\rm loc}(\overline{\mathbb{R}^N_+})$ and
	\begin{equation}\label{gradient_bound_strips}
		|\nabla u| \in L^\infty(\Sigma_\lambda) \quad\text{ for all } \lambda>0,
	\end{equation}
	where the sets $\Sigma_\lambda := \{(x',x_N)\in\mathbb{R}^N \mid 0<x_N<\lambda \}$ are called strips. By the mean value theorem and the Dirichlet boundary condition, one can check that assumption \eqref{gradient_bound_strips} implies
	\begin{equation}\label{bound_strips}
		u\in L^\infty(\Sigma_\lambda) \text{ for all } \lambda>0 \text{ and } \lim_{x_N\to0^+} u(x',x_N) = 0 \text{ uniformly in } x'\in\mathbb{R}^{N-1}_+.
	\end{equation}
	Conversely, if $u$ satisfies \eqref{bound_strips}, then $u$ and $f(u)$ are bounded in each strip $\Sigma_\lambda$ and the gradient bound \eqref{gradient_bound_strips} can be obtained via standard elliptic estimates (see \cite{MR709038,MR727034}).
	
	To the best of our knowledge, problem \eqref{main} with a singular nonlinearity such as \eqref{singular} has not been well studied in the literature, except for the case $p=2$ (see \cite{MR4753083,2024arXiv240403343M,2024arXiv240900365L}). Compared to the regular nonlinearity, the difficulty in studying this problem is magnified by the fact that the gradient of solutions to singular problems usually exhibits singularity on the boundary of the domain (see \cite{MR1037213}). Hence it is not appropriate to enforce assumptions $u \in C^{1,\alpha}_{\rm loc}(\overline{\mathbb{R}^N_+})$ as well as \eqref{gradient_bound_strips} when studying \eqref{main} or \eqref{singular}. Instead, we are interested in qualitative properties of weak solutions $u \in C^{1,\alpha}_{\rm loc}(\mathbb{R}^N_+) \cap C(\overline{\mathbb{R}^N_+})$ which satisfy assumption \eqref{bound_strips}. By weak solutions, we mean
	\[
	\int_{\mathbb{R}^N_+} (|\nabla u|^{p-2} \nabla u, \nabla \varphi) = \int_{\mathbb{R}^N_+} f(u) \varphi \quad\text{ for all } \varphi\in C^1_c(\mathbb{R}^N_+).
	\]
	On the other hand, the $C^{1,\alpha}_{\rm loc}(\mathbb{R}^N_+)$ regularity is a natural one taking into account the standard regularity results in \cite{MR709038,MR727034,MR969499}.	
	Moreover, we have the following criterion for the uniform convergence of $u$ near the boundary mentioned in \eqref{bound_strips}.
	\begin{proposition}\label{prop:zero_convergence}
		Assume that $p>1$ and $f:(0,+\infty)\to\mathbb{R}$ is a locally Lipschitz continuous function such that $f$ is strictly decreasing on $(0,\rho)$ for some $\rho>0$.
		Let $u \in C^{1,\alpha}_{\rm loc}(\mathbb{R}^N_+) \cap C(\overline{\mathbb{R}^N_+})$ be a solution to \eqref{main}. If $\|u\|_{L^\infty(\Sigma_{\overline\lambda})} < \rho$ for some $\overline\lambda>0$, then
		\[
		\lim_{x_N\to0^+} u(x',x_N) = 0 \text{ uniformly in } x'\in\mathbb{R}^{N-1}_+.
		\]
	\end{proposition}
	As we can see in problem \eqref{singular}, the assumption that the nonlinearity is strictly decreasing near zero is reasonable for singular problems. To state our main results, we denote by $Z_f$ the set of zeros of $f$ in $(0,+\infty)$, that is,
	\[
	Z_f:=\{t\in(0,+\infty) \mid f(t)=0\}.
	\]
	
	Our main idea in tackling \eqref{main} or \eqref{singular} is to isolate a small strip $\Sigma_{\tilde{\lambda}}$ where the singular phenomenon may appear. We show that $u$ is monotone increasing in this strip via a weak comparison principle. Then in the remaining area $\mathbb{R}^N_+\setminus\Sigma_{\tilde{\lambda}}$, we basically combine the techniques in \cite{MR3303939,MR2886112,MR4439897,MR3752525,MR3118616} to address the monotonicity of solutions to \eqref{main} in the full domain. Our first result is the following:.
	\begin{theorem}\label{th:monotonicity} Assume that $f:(0,+\infty)\to\mathbb{R}$ is a locally Lipschitz continuous function such that
		\begin{itemize}
			\item[(i)] $\lim_{t\to0^+} f(t) > 0$,
			\item[(ii)] $f$ is strictly decreasing on $(0,t_0)$ for some $t_0>0$,
			\item[(iii)] either $\frac{2N+2}{N+2} < p < 2$ and $Z_f$ is a discrete set, or $p>1$ and $f(t)>0$ for $t>0$.
		\end{itemize}
		Let $u \in C^{1,\alpha}_{\rm loc}(\mathbb{R}^N_+) \cap C(\overline{\mathbb{R}^N_+})$ be a solution to problem \eqref{main} satisfying \eqref{bound_strips}.
		Then $u$ is monotone increasing in $x_N$.
	\end{theorem}
	
	\begin{remark}\label{rem:1}
		Once the monotonicity of $u$ is obtained, we may argue as in \cite[Lemma 14]{10.1007/s11118-024-10157-1} to further derive
		\[
		\frac{\partial u}{\partial x_N} > 0 \quad\text{ in } \mathbb{R}^N_+\setminus(Z_{f(u)}\cap Z_u)
		\]
		provided that $p>\frac{2N+2}{N+2}$, where $Z_{f(u)}:=\{x\in\mathbb{R}^N_+ \mid f(u(x))=0\}$ and $Z_u:=\{x\in\mathbb{R}^N_+ \mid |\nabla u(x)|=0\}$ (see also \cite[Theorem 1.2]{MR2188744} and \cite[Theorem 1.1]{MR4439897}).
	\end{remark}
	
	If $g:[0,+\infty)\to\mathbb{R}$ is locally Lipschitz continuous, then $t\mapsto \frac{1}{t^\gamma} + g(t)$ is strictly decreasing on $(0,t_0)$ for some $t_0>0$.
	If we further assume that either $\frac{2N+2}{N+2} < p < 2$ and $\{t\in(0,+\infty) \mid \frac{1}{t^\gamma} + g(t)=0\}$ is a discrete set, or $\frac{1}{t^\gamma} + g(t)>0$ for $t>0$. Then Theorem \ref{th:monotonicity} and Proposition \ref{prop:zero_convergence} indicate that every solution $u \in C^{1,\alpha}_{\rm loc}(\mathbb{R}^N_+) \cap C(\overline{\mathbb{R}^N_+})$ to problem \eqref{singular} with $u\in L^\infty(\Sigma_\lambda)$ for all $\lambda>0$ and $\|u\|_{L^\infty(\Sigma_{\overline\lambda})} < t_0$ for some $\overline\lambda>0$ is monotone increasing in $x_N$. By Remark \ref{rem:1}, we further derive $\frac{\partial u}{\partial x_N} > 0$ provided that $p>\frac{2N+2}{N+2}$ and $\frac{1}{t^\gamma} + g(t)>0$ for $t>0$.
	
	In fact, due to the appearance of the explicit singular term $\frac{1}{t^\gamma}$, we would expect stronger results on the monotonicity of solutions to problem \eqref{singular}, in particular, on their behavior near the boundary. In \cite{MR4044739} Esposito and Sciunzi showed that inward directional derivatives near the boundary of solutions to problem \eqref{singular} posed in a bounded domain is necessarily positive for all $p>1$ regardless of the sign of $t\mapsto\frac{1}{t^\gamma} + g(t)$. In the next theorem, we not only prove a similar claim and a monotonicity result for \eqref{singular} but also provide a sharp estimate on derivatives, which indicates that they must blow up at a proper rate near the boundary.
	\begin{theorem}\label{th:monotonicity2} Assume that $p>1$, $\gamma>1$ and $g:[0,+\infty)$ is a locally Lipschitz continuous function.
		Let $u \in C^{1,\alpha}_{\rm loc}(\mathbb{R}^N_+) \cap C(\overline{\mathbb{R}^N_+})$ be a solution to problem \eqref{singular} with $u\in L^\infty(\Sigma_\lambda)$ for all $\lambda>0$. Then for every $\beta\in(0,1)$, there exist $c_1,c_2,\lambda_0>0$ such that
		\begin{equation}\label{gradient_blowup}
			c_1 x_N^{\frac{1-\gamma}{\gamma+p-1}} < \frac{\partial u(x)}{\partial \eta} < c_2 x_N^{\frac{1-\gamma}{\gamma+p-1}} \quad\text{ in } \Sigma_{\lambda_0}
		\end{equation}
		for all $\eta\in \mathbb{S}^{N-1}_+$ with $(\eta,e_N) \ge \beta$, where $\mathbb{S}^{N-1}_+:=\mathbb{R}^N_+ \cap \partial B_1(0)$ and $e_N:=(0,\dots,0,1)$. If we further assume that either $\frac{2N+2}{N+2} < p < 2$ and $\{t\in(0,+\infty) \mid \frac{1}{t^\gamma} + g(t)=0\}$ is a discrete set, or $\frac{1}{t^\gamma} + g(t)>0$ for $t>0$, then $u$ is monotone increasing in $x_N$. Moreover, we have $\frac{\partial u}{\partial x_N} > 0$ in $\mathbb{R}^N_+$ provided that $p>\frac{2N+2}{N+2}$ and $\frac{1}{t^\gamma} + g(t)>0$ for $t>0$.
	\end{theorem}
	
	The one-dimensional (1D) symmetry of solutions to problem \eqref{main}, commonly referred to as a rigidity result in the literature, has been explored in the semilinear case where $p = 2$ by Berestycki, Caffarelli, and Nirenberg \cite{MR1470317,MR1655510}, Angenent \cite{MR794096}, and Clément and Sweers \cite{MR937538}. However, this topic is not well understood in the case $p\ne2$. For $p \ne 2$, there are some results in lower dimensions under the condition that the solutions and their gradients are bounded, as seen in \cite{MR2654242} for the case $N = 2$, $p > \frac{3}{2}$, and \cite{MR3118616,MR2886112} for $N = 3$, $p > \frac{8}{5}$. In higher dimensions, Du and Guo \cite{MR2056284} have addressed the 1D symmetry of bounded positive solutions to \eqref{main}, assuming the condition that $f(t)>0$ for all $0<t<1$, $f(t)<0$ for all $t>1$ and $f(t) \ge c_0 t^{p-1}$ in $(0,\sigma)$ for some $c_0,\sigma>0$. Under this assumption, the uniqueness of bounded solutions to \eqref{main} is ensured by the method of sub-super solutions, leading to the symmetry of solutions due to the symmetry of the domain. A more general rigidity result for bounded solutions was obtained recently in \cite{2024arXiv240904804L} using a similar method. We also mention the excellent work \cite{MR4377321}, where various maximum and comparison principles were exploited with the moving plane method to show the 1D symmetry of bounded solutions to a $p$-Laplace equation in the whole space $\mathbb{R}^N$ with uniform limits. Later, these analytic tools were resorted to a more convenient sliding method to study similar problems in the whole and a half-space \cite{MR4771444,MR4635360}.	
	In this paper, we exploit such a sliding method to prove the following rigidity result.
	\begin{theorem}\label{th:rigidity1}
		Assume that $p>\frac{2N+2}{N+2}$ and $f:(0,+\infty)\to\mathbb{R}$ is a locally Lipschitz continuous function such that
		\begin{itemize}
			\item[(i)] $\liminf_{t\to0^+} \frac{f(t)}{t^{p-1}} > 0$,
			\item[(ii)] $f(t)>0$ for $t>0$,
			\item[(iii)] $f$ is strictly decreasing on $(t_0,+\infty)$ for some $t_0>0$.
		\end{itemize}
		Let $u \in C^{1,\alpha}_{\rm loc}(\mathbb{R}^N_+) \cap C(\overline{\mathbb{R}^N_+})$ be a solution to problem \eqref{main} satisfying \eqref{bound_strips} and
		\begin{equation}\label{uni_gradient_zero}
			\lim_{x_N\to+\infty} |\nabla u(x',x_N)|=0 \text{ uniformly in } x'\in\mathbb{R}^{N-1}_+.
		\end{equation}		
		Then $u$ depends only on $x_N$ and is monotone increasing in $x_N$.
	\end{theorem}	
	Theorem \eqref{th:rigidity1} is useful for problems with positive nonlinearities. For sign-changing nonlinearities, we have the following partial result.
	\begin{theorem}\label{th:rigidity2}
		Assume that $\frac{2N+2}{N+2} < p < 2$ and $f:(0,+\infty)\to\mathbb{R}$ is a locally Lipschitz continuous function such that
		\begin{itemize}
			\item[(i)] $\liminf_{t\to0^+} \frac{f(t)}{t^{p-1}} > 0$,
			\item[(ii)] $Z_f$ is a nonempty discrete set,
			\item[(iii)] $f$ is strictly decreasing on $(t_0,+\infty)$ and $Z_f\cap(t_0,+\infty)=\emptyset$ for some $t_0>0$,
		\end{itemize}
		Let $u \in C^{1,\alpha}_{\rm loc}(\mathbb{R}^N_+) \cap C(\overline{\mathbb{R}^N_+})$ be a solution to problem \eqref{main} satisfying \eqref{bound_strips}, \eqref{uni_gradient_zero} and
		\begin{equation}\label{uni_large}
			\liminf_{x_N\to+\infty} u(x', x_N) > t_0 \text{ uniformly in } x'\in\mathbb{R}^{N-1}.
		\end{equation}
		Then $u$ depends only on $x_N$ and is monotone increasing in $x_N$.
	\end{theorem}
	
	As a consequence of Theorem \ref{th:rigidity1}, we have the following result for nonlinearities that change sign once.
	\begin{proposition}\label{prop:rigidity}
		Assume that $p > \frac{2N+2}{N+2}$ and $f:(0,+\infty)\to\mathbb{R}$ is a locally Lipschitz continuous function such that
		\begin{itemize}
			\item[(i)] $\liminf_{t\to0^+} \frac{f(t)}{t^{p-1}} > 0$,
			\item[(ii)] $f(t) > 0$ in $(0,t_0)$, $f(t) < 0$ in $(t_0,+\infty)$ for some $t_0>0$,
			\item[(iii)] $f$ is strictly decreasing on $(t_0-\delta,t_0]$ for some $\delta>0$.
		\end{itemize}
		Let $u \in C^{1,\alpha}_{\rm loc}(\mathbb{R}^N_+) \cap C(\overline{\mathbb{R}^N_+})$ be a bounded solution to problem \eqref{main} satisfying \eqref{bound_strips}. Then $u$ depends only on $x_N$ and is monotone increasing in $x_N$. Moreover, $0<u<t_0$ in $\mathbb{R}^N_+$ and $\lim_{x_N\to0} u(x',x_N)= t_0$ uniformly in $x'\in\mathbb{R}^{N-1}$.
		
		If we further assume
		\begin{equation}\label{bounded_cond}
			\int_{t_1}^{\infty} \left(\int_{t_1}^{s}|h(t)|dt \right)^{-\frac{1}{p}} ds < \infty
		\end{equation}
		for some $t_1\ge t_0$ and some continuous function $h$ such that $f(t) \le h(t) < 0$ for all $t\in[M,+\infty)$, then the same conclusion holds for all solutions $u \in C^{1,\alpha}_{\rm loc}(\mathbb{R}^N_+) \cap C(\overline{\mathbb{R}^N_+})$ satisfying \eqref{bound_strips}.
	\end{proposition}
	
	Now we consider a special case where $f$ is strictly decreasing in the whole domain. In such a case, we can remove the restriction on $p$.
	\begin{theorem}\label{th:decreasing}
		Assume that $p>1$ and $f:(0,+\infty)\to\mathbb{R}$ is a locally Lipschitz continuous function and
		\begin{itemize}
			\item[(i)] $\lim_{t\to0^+} f(t) > 0$,
			\item[(ii)] $f$ is strictly decreasing on $(0,+\infty)$.
		\end{itemize}
		Let $u \in C^{1,\alpha}_{\rm loc}(\mathbb{R}^N_+) \cap C(\overline{\mathbb{R}^N_+})$ be a solution to problem \eqref{main} with $u\in L^\infty(\Sigma_\lambda)$ for all $\lambda>0$.
		Then $u$ is monotone increasing in $x_N$.
		
		Moreover, if either \eqref{uni_gradient_zero} holds, or $\limsup_{t\to+\infty} f'(t) < 0$ and $|\nabla u| \in L^\infty(\mathbb{R}^N_+\setminus\Sigma_{\overline\lambda})$ for some $\overline\lambda>0$, then $u$ depends only on $x_N$.
	\end{theorem}
	Lastly, we focus our attention on problem \eqref{singular} with $g\equiv0$ and $\gamma>0$. This problem is of particular interest due to its application in blow-up analysis (see \cite{MR4044739}). For this problem, we have the following classification result for $1<p<N$:
	\begin{theorem}\label{th:pure}
		Assume $1<p< N$. Let $\gamma>1$ and let $u \in C^{1,\alpha}_{\rm loc}(\mathbb{R}^N_+) \cap C(\overline{\mathbb{R}^N_+})$ be a solution to the problem
		\begin{equation}\label{pure}
			\begin{cases}
				-\Delta_p u = \dfrac{1}{u^\gamma} & \text{ in } \mathbb{R}^N_+,\\
				u>0 & \text{ in } \mathbb{R}^N_+,\\
				u=0 & \text{ on } \partial\mathbb{R}^N_+
			\end{cases}
		\end{equation}
		with
		\begin{equation}\label{bound_strip}
			u\in L^\infty(\Sigma_{\overline\lambda}) \text{ for some } \overline\lambda>0.
		\end{equation}
		Then $u$ is monotone increasing in $x_N$ and
		\begin{equation}\label{asymptotic}
			\limsup_{x_N\to+\infty} \frac{u(x',x_N)}{x_N}<+\infty \text{ uniformly in } x'\in\mathbb{R}^{N-1}_+.
		\end{equation}
		If we further assume that $u$ is sublinear in the sense that	
		\begin{equation}\label{sublinear}
			\lim_{x_N\to+\infty} \frac{u(x',x_N)}{x_N}=0 \text{ uniformly in } x'\in\mathbb{R}^{N-1}_+,
		\end{equation}
		then
		\[
		u(x) \equiv \left[\frac{(\gamma+p-1)^p}{p^{p-1}(p-1)(\gamma-1)}\right]^\frac{1}{\gamma+p-1}	x_N^\frac{p}{\gamma+p-1}.
		\]
		
		If else $0<\gamma<1$, then \eqref{pure} has no solution $u \in C^{1,\alpha}_{\rm loc}(\mathbb{R}^N_+) \cap C(\overline{\mathbb{R}^N_+})$ satisfying \eqref{bound_strip}.
	\end{theorem}
	To prove Theorem \ref{th:pure}, we extend some estimates in \cite{MR4753083} from $p=2$ to $p\ne2$, then we apply Theorem \ref{th:decreasing}.
	Theorem \ref{th:pure} improves a classification result in \cite[Theorem 1.2]{MR4044739}, where the exact asymptotic behavior
	\[
	c x_N^\frac{p}{\gamma+p-1} \le u(x) \le C x_N^\frac{p}{\gamma+p-1} \quad\text{ in } \mathbb{R}^N_+
	\]
	is assumed instead of \eqref{sublinear}. We stress that assumption \eqref{sublinear} is sharp in the sense that solutions which do not satisfy \eqref{sublinear} do exist (see Theorem \ref{th:1D} below). However, we cannot classify all such solutions without a priori assumption \eqref{sublinear}. We recall that all solutions to problem \eqref{pure} when $p=2$ were classified in \cite{MR4753083} without assumption \eqref{sublinear} and in \cite{2024arXiv240403343M} without also \eqref{bound_strip}. The key tools available in the case $p=2$ to study such a problem are a maximum principle for bounded solutions in unbounded domains (see \cite[Lemma 2.1]{MR1470317}) and the Kelvin transform. It seems not to be easy to extend such a result to the case $p\ne2$ due to the lack of the Kelvin transform for the $p$-Laplacian and the nonlinear nature of this operator. Nevertheless, such a transform is available for the $N$-Laplacian besides the Laplacian one. For this reason, we can utilize the Kelvin transform to classify all solutions to \eqref{pure} with $p=N$ without assumptions \eqref{bound_strip} and \eqref{sublinear}.
	\begin{theorem}\label{th:pure_N}
		Let $\gamma>1$ and let $u \in C^{1,\alpha}_{\rm loc}(\mathbb{R}^N_+) \cap C(\overline{\mathbb{R}^N_+})$ be a solution to the problem		
		\begin{equation}\label{pure_N}
			\begin{cases}
				-\Delta_N u = \dfrac{1}{u^\gamma} & \text{ in } \mathbb{R}^N_+,\\
				u>0 & \text{ in } \mathbb{R}^N_+,\\
				u=0 & \text{ on } \partial\mathbb{R}^N_+
			\end{cases}
		\end{equation}
		Then either $u(x)\equiv v_0(x_N)$ or $u(x)\equiv \lambda^{-\frac{N}{N+\gamma-1}} v_1(\lambda x_N)$ for some $\lambda>0$, where
		\[
		v_0(t) := \left[\frac{(N+\gamma-1)^N}{N^{N-1}(N-1)(\gamma-1)}\right]^\frac{1}{N+\gamma-1}	t^\frac{N}{N+\gamma-1}
		\]
		and $v_1$ is uniquely determined by
		\[	
		\int_{0}^{v_1(t)} \left(1 + \frac{s^{1-\gamma}}{\gamma-1}\right)^{-\frac{1}{N}} ds = \left(\frac{N}{N-1}\right)^\frac{1}{N}t \quad\text{ for all } t\ge0.
		\]
		
		If else $0<\gamma\le1$, then \eqref{pure} has no solution in $C^{1,\alpha}_{\rm loc}(\mathbb{R}^N_+) \cap C(\overline{\mathbb{R}^N_+})$.
	\end{theorem}
	
	Moreover, we can classify all solutions to \eqref{pure_N} for all $p>1$ in dimension one.
	\begin{theorem}\label{th:1D}
		Assume $p>1$. Let $\gamma>1$ and let $u \in C^{1,\alpha}_{\rm loc}(\mathbb{R}_+) \cap C(\overline{\mathbb{R}_+})$ be a solution to the problem
		\begin{equation}\label{1D}
			\begin{cases}
				-(|v'|^{p-2} v')' = \dfrac{1}{v^\gamma} & \text{ in } \mathbb{R}_+,\\
				v(t)>0 & \text{ in } \mathbb{R}_+,\\
				v(0)=0.
			\end{cases}
		\end{equation}
		Then either
		\[
		v(t) \equiv v_0(t) := \left[\frac{(\gamma+p-1)^p}{p^{p-1}(p-1)(\gamma-1)}\right]^\frac{1}{\gamma+p-1}	t^\frac{p}{\gamma+p-1}
		\]
		or
		\[
		v(t) \equiv \lambda^{-\frac{p}{\gamma+p-1}} v_1(\lambda t) \quad\text{ for some } \lambda>0,
		\]
		where $v_1$ is uniquely determined by
		\[	
		\int_{0}^{v_1(t)} \left(1 + \frac{s^{1-\gamma}}{\gamma-1}\right)^{-\frac{1}{p}} ds = \left(\frac{p}{p-1}\right)^\frac{1}{p}t \quad\text{ for all } t\ge0.
		\]
		
		If else $0<\gamma\le1$, then \eqref{1D} has no solution in $C^{1,\alpha}_{\rm loc}(\mathbb{R}_+) \cap C(\overline{\mathbb{R}_+})$.
	\end{theorem}	
	The proof of Theorem \ref{th:1D} combines PDE and ODE techniques. Some special cases were obtained in \cite[Proposition 2.4]{MR4044739} and \cite[Theorem 11]{MR4753083} by different methods. We expect that all solutions to \eqref{pure} with $1<p<N$ in higher dimensions without restriction \eqref{sublinear} are indeed 1D and given by Theorem \ref{th:1D}. We leave it as an open question.
	
	The rest of this paper is organized as follows. In Section \ref{sect2} we recall two versions of strong comparison principles that will be used later. Then we prove a weak comparison principle for strips, prove Proposition \ref{prop:zero_convergence}, and provide some a priori bounds for solutions. In Section \ref{sect3}, we use the method of moving planes to prove Theorem \ref{th:monotonicity} and use a scaling technique to prove Theorem \ref{th:monotonicity2}. In Section \ref{sect4}, we exploit some other comparison principles and the sliding method to prove the 1D symmetry of solutions stated in Theorem \ref{th:rigidity1}, \ref{th:rigidity2} and Proposition \ref{prop:rigidity}. In Section \ref{sect5}, we focus on problems whose nonlinearity is strictly decreasing in the whole $(0,+\infty)$ and we provide proofs for Theorems \ref{th:decreasing}, \ref{th:1D}, \ref{th:pure} and \ref{th:pure_N}.
	
	\section{Preliminaries}\label{sect2}
	
	We always assume that $\Omega$ is a connected domain of $\mathbb{R}^N$ and $f$ is a locally Lipschitz continuous function.
	In the quasilinear case, the maximum principle is not equivalent to the comparison one. Therefore, we also need to recall the classical version of the strong comparison principle for $p$-Laplace equations.
	\begin{theorem}[Strong comparison principle 1 \cite{MR1632933}]\label{th:cscp}
		Let $u,v\in C^1(\Omega)$ be two solutions to
		\[
		-\Delta_p w = f(w) \quad\text{ in } \Omega
		\]
		such that $u \le v$ in $\Omega$, with $p>1$ and let
		\[
		Z=\{x\in\Omega \mid |\nabla u(x)|+|\nabla v(x)|=0\}.
		\]
		If $x_0\in\Omega\setminus Z$ and $u(x_0)=v(x_0)$, then $u=v$ in the connected component of $\Omega\setminus Z$ containing $x_0$.
	\end{theorem}
	
	Theorem \ref{th:cscp} only holds far from the degenerate set. Now we present a result that holds, under stronger assumptions, on the entire domain $\Omega$.
	\begin{theorem}[Strong comparison principle 2 \cite{MR2096703}]\label{th:scp}
		Let $u,v\in C^1(\overline{\Omega})$ be two solutions to
		\[
		-\Delta_p w = f(w) \quad\text{ in } \Omega,
		\]
		where $p > \frac{2N+2}{N+2}$. Assume $u\le v$ in $\Omega$ and at least one of the following two conditions holds:
		\begin{itemize}
			\item[(i)] either
			\[
			f(u(x)) > 0 \quad\text{ in } \overline{\Omega}
			\]
			or
			\[
			f(u(x)) < 0 \quad\text{ in } \overline{\Omega},
			\]
			\item[(ii)] either
			\[
			f(v(x)) > 0 \quad\text{ in } \overline{\Omega}
			\]
			or
			\[
			f(v(x)) < 0 \quad\text{ in } \overline{\Omega}.
			\]
		\end{itemize}
		Then either $u=v$ in $\Omega$ or $u<v$ in $\Omega$.
	\end{theorem}
	
	In the situations where the above strong comparison principles do not apply, we will make use of the following weak sweeping principle by Dancer and Du.
	\begin{theorem}[Weak sweeping principle \cite{MR3058211}]\label{th:wsp}
		Suppose that $\Omega$ is a bounded smooth domain in $\mathbb{R}^N$, $(x, s)\mapsto h(x, s)$ is measurable in $x \in \Omega$,
		continuous in $s$, and for each finite interval $J$, there exists a continuous increasing function $L(s)$ such
		that $s\mapsto h(x, s) + L(s)$ is nondecreasing in $s$ for $s \in J$ and $x \in \Omega$. Let $u_t$ and $v_t$, $t \in [t_1, t_2]$, be functions in $W^{1,p}(\Omega) \cap C(\overline \Omega)$ and satisfy in the weak sense,
		\[
		\begin{cases}
			-\Delta_p u_t \ge h(x,u_t) + \varepsilon_1(t) & \text{ in } \Omega,\\
			-\Delta_p v_t \le h(x,v_t) - \varepsilon_2(t) & \text{ in } \Omega,\\
			u_t \ge v_t + \varepsilon & \text{ on } \partial\Omega,\\
		\end{cases}
		\]
		for all $t \in [t_1, t_2]$, where
		\[
		\varepsilon_1(t) + \varepsilon_2(t) \ge \varepsilon > 0.
		\]
		Moreover, suppose that $u_{t_0} \ge v_{t_0}$ in $\Omega$ for some $t_0 \in [t_1, t_2]$ and $t \mapsto u_t$, $t \mapsto v_t$ are continuous from the finite closed interval $[t_1, t_2]$ to $C(\overline \Omega)$. Then
		\[
		u_t \ge v_t \text{ in } \Omega \text{ for all } t \in [t_1, t_2].
		\]
	\end{theorem}
	
	The statement of Theorem \ref{th:wsp} is taken from \cite{MR3058211}. The proof of this theorem is almost identical to that of \cite[Lemma 2.7]{MR1955278}.
	
	Throughout the paper, we denote generic positive constants by $C$ (with dependent subscripts in some cases) and they will be allowed to vary within a single line or formula. We also denote by $f^+$ the positive part of a function $f$, that is, $f^+ = \max\{f,0\}$ and by $B_R$, $B_R'$ the open balls of radius $R>0$ centered at the origin in $\mathbb{R}^N$ and $\mathbb{R}^{N-1}$, respectively. For brevity, we drop $dx$ in the integral notations when it is clear from the context.
	
	\subsection{Weak comparison principle for strips}
	The aim of this section is the following weak comparison principle, which can be applied to problems with singular nonlinearities.
	\begin{proposition}\label{prop:wcp}
		Let $f:(0,+\infty)\to\mathbb{R}$ be a locally Lipschitz continuous function such that $f$ is strictly decreasing on $(0,t_0)$ for some $t_0>0$ and let $\Sigma := \Sigma_{\lambda}$ with $\lambda>0$. Assume that $u, v \in C^{1,\alpha}_{\rm loc}(\Sigma) \cap C(\overline{\Sigma})$ satisfy
		\begin{equation}\label{ky}
			\begin{cases}
				-\Delta_p u = f(u) &\text{ in } \Sigma,\\
				-\Delta_p v = f(v) &\text{ in } \Sigma,\\
				0<u<t_0 &\text{ in } \Sigma,\\
				v>\delta &\text{ in } \Sigma,\\
				u \le v &\text{ on } \partial\Sigma
			\end{cases}
		\end{equation}
		for some $\delta>0$. Then $u\le v$ in $\Sigma$.
	\end{proposition}
	
	To prove Proposition \ref{prop:wcp}, we need the following elementary lemma, which would appear somewhere in the literature. However, we cannot find a suitable reference. Therefore, we provide a proof for the reader's convenience.
	\begin{lemma}\label{lem:k}
		Let $-\infty \le m < M \le +\infty$.
		If $g:(m,M)\to\mathbb{R}$ is continuous and strictly decreasing, then
		\begin{equation}\label{k1}
			\sup_{\substack{t_1,t_2\in[l_1,l_2]\\t_2-t_1\ge\varepsilon}} \frac{g(t_2)-g(t_1)}{t_2-t_1} < 0
		\end{equation}
		for every $\varepsilon>0$ and every interval $[l_1,l_2] \subset (m,M)$ with $l_2-l_1\ge\varepsilon$.
		
		If we further assume that $-\infty < m < M=+\infty$ and $g$ is differentiable with
		\begin{equation}\label{neg_deri}
			\limsup_{t\to+\infty} g'(t) < 0,
		\end{equation}
		then
		\begin{equation}\label{k2}
			\sup_{\substack{t_1,t_2\in(m,+\infty)\\t_2-t_1\ge\varepsilon}} \frac{g(t_2)-g(t_1)}{t_2-t_1} < 0
		\end{equation}
		for all $\varepsilon>0$.
	\end{lemma}
	
	\begin{proof}
		Assume by contradiction that \eqref{k1} does not hold, then there exist $\varepsilon>0$, an interval $[l_1,l_2] \subset (m,M)$ with $l_2-l_1\ge\varepsilon$ and two sequences $(b_n),(c_n)$ such that $l_1\le b_n\le c_n\le l_2$, $c_n-b_n\ge\varepsilon$ and
		\begin{equation}\label{f'n}
			\frac{g(c_n)-g(b_n)}{c_n-b_n} \to 0.
		\end{equation}
		Up to a subsequence, $b_n\to b$ and $c_n\to c$ with $b,c\in[l_1,l_2]$ such that $c-b\ge\varepsilon$. Consequently, \eqref{f'n} implies
		\[
		\frac{g(c)-g(b)}{c-b} = 0.
		\]
		This is a contradiction with the assumption that $g$ is strictly decreasing. Hence \eqref{k1} is proved.
		
		Now we assume that \eqref{k2} does not hold for some $\varepsilon>0$. Then we can find three sequences $(a_n),(b_n),(c_n)$ such that $m\le b_n\le a_n\le c_n$, $c_n-b_n\ge\varepsilon$ and
		\begin{equation}\label{xn}
			g'(a_n) = \frac{g(c_n)-g(b_n)}{c_n-b_n} \to 0.
		\end{equation}
		This implies that $(a_n)$ is bounded. Therefore, $(b_n)$ is also bounded. Passing to a subsequence, we may assume $b_n\to b \in[m,+\infty)$ and $c_n\to c \in[m,+\infty]$ with $c-b\ge \varepsilon$ if $c<+\infty$. Consequently, \eqref{xn} implies
		\[
		\frac{g(c)-g(b)}{c-b} = 0 \quad\text{ if } c<+\infty,
		\]
		\[
		\frac{g(c_n)-g(b)}{c_n-b} \to 0 \quad\text{ if } c_n\to+\infty.
		\]
		However, the former contradicts the fact that $g$ is strictly decreasing, while the latter contradicts \eqref{neg_deri}. This completes the proof.
	\end{proof}
	
	\begin{remark}
		Due to Lemma \ref{lem:k}, the assumptions (1.5) in \cite{MR4635360} and (F2) in \cite{MR4771444} can be reduced to the requirement that $f$ is strictly decreasing in the corresponding intervals.
	\end{remark}
	
	For later use, we recall the following elementary inequalities
	\begin{align}
		(|\xi|^{p-2}\xi - |\xi'|^{p-2}\xi', \xi -\xi') &\ge C_1 (|\xi| + |\xi'|)^{p-2} |\xi -\xi'|^2, \label{p_ge}\\
		\left||\xi|^{p-2}\xi - |\xi'|^{p-2}\xi'\right| &\le C_2 (|\xi| + |\xi'|)^{p-2} |\xi -\xi'|, \label{p_le}
	\end{align}
	which hold for all $\xi, \xi' \in\mathbb{R}^N$ with $|\xi|+|\xi'|>0$, where $p>1$ and $C_1,C_2>0$ depend only on $N$ and $p$.
	
	Now we provide a proof of Proposition \ref{prop:wcp}.
	
	\begin{proof}[Proof of Proposition \ref{prop:wcp}]
		For each $R>0$, let $\varphi_R \in C^1(\mathbb{R}^{N-1})$ be a standard cutoff function, which satisfies
		\begin{equation}\label{phi}
			\begin{cases}
				0 \le \varphi_R \le 1 &\text{ in } \mathbb{R}^{N-1},\\
				\varphi_R = 1 &\text{ in } B_R',\\
				\varphi_R = 0 &\text{ in } \mathbb{R}^{N-1} \setminus B_{2R}',\\
				|\nabla\varphi_R| \le \frac{2}{R} &\text{ in } B_{2R}' \setminus B_R',
			\end{cases}
		\end{equation}
		where we recall that $B_r'$ is the ball in $\mathbb{R}^{N-1}$ of radius $r$ and center at the origin.
		
		Fix some $\alpha>N-2$ and $\varepsilon>0$. Then we set $w=(u-v-\varepsilon)^+$ and
		\[
		\psi(x) := w^{\alpha}(x) \varphi_R^{\alpha+1}(x') \chi_{\Sigma}(x).
		\]
		Since the support of $\psi$ is compactly contained in $\Sigma$, we can use $\psi$ as a test function in the equations $-\Delta_p u = f(u)$ and $-\Delta_p v = f(v)$. Then subtracting, we obtain
		\begin{equation}\label{subtract}
			\begin{aligned}
				&\alpha \int_{\Sigma} (|\nabla u|^{p-2}\nabla u - |\nabla v|^{p-2}\nabla v, \nabla w) w^{\alpha-1} \varphi_R^{\alpha+1} \\
				&\qquad+ (\alpha+1) \int_{\Sigma} (|\nabla u|^{p-2}\nabla u - |\nabla v|^{p-2}\nabla v, \nabla \varphi_R) w^{\alpha} \varphi_R^\alpha\\
				&= \int_{\Sigma} \left(f(u) - f(v)\right) w^{\alpha} \varphi_R^{\alpha+1}.
			\end{aligned}
		\end{equation}
		
		Using \eqref{p_ge} and \eqref{p_le}, we deduce from \eqref{subtract} that
		\begin{equation}\label{t0}
			\begin{aligned}
				&\alpha C_1 \int_{\Sigma} (|\nabla u| + |\nabla v|)^{p-2} |\nabla w|^2 w^{\alpha-1} \varphi_R^{\alpha+1}\\
				&\le (\alpha+1)C_2 \int_{\Sigma} (|\nabla u| + |\nabla v|)^{p-2} |\nabla w| |\nabla \varphi_R| w^{\alpha} \varphi_R^\alpha + \int_{\Sigma} \left(f(u) - f(v)\right) w^{\alpha} \varphi_R^{\alpha+1}\\
				&\le (\alpha+1)C_2 \int_{\Sigma} (|\nabla u| + |\nabla v|)^{p-1} |\nabla \varphi_R| w^{\alpha} \varphi_R^\alpha + \int_{\Sigma} \left(f(u) - f(v)\right) w^{\alpha} \varphi_R^{\alpha+1}.
			\end{aligned}
		\end{equation}
		
		In the set ${\Sigma}\cap\{w > 0\}$, we have
		\begin{equation}\label{bound0}
			\delta < v < v+\varepsilon < u < t_0.
		\end{equation}
		Since $f$ is strictly decreasing on $(0,t_0)$, Lemma \ref{lem:k} gives
		\begin{equation}\label{t2}
			f(u) - f(v) \le - C_\varepsilon (u-v) \le - C_\varepsilon w \quad\text{ in } {\Sigma}\cap\{w > 0\}
		\end{equation}
		for some $C_\varepsilon>0$.
		On the other hand, from \eqref{bound0} we have that $f(u)$ and $f(v)$ are bounded. Hence the standard gradient estimate yields
		\begin{equation}\label{t3}
			|\nabla u| < C_0 \quad\text{ and }\quad |\nabla v| < C_0 \quad\text{ in } {\Sigma}\cap\{w > 0\}.
		\end{equation}
		
		Substituting \eqref{t2}, \eqref{t3} into \eqref{t0}, we obtain
		
		\begin{align*}
			&\alpha C_1 \int_{\Sigma} (|\nabla u| + |\nabla v|)^{p-2} |\nabla w|^2 w^{\alpha-1} \varphi_R^{\alpha+1}\\
			&\le (\alpha+1)C_2(2C_0)^{p-1} \int_{\Sigma} |\nabla \varphi_R| w^{\alpha} \varphi_R^\alpha - C_\varepsilon \int_{\Sigma} w^{\alpha+1} \varphi_R^{\alpha+1}.
		\end{align*}
		
		Applying the weighted Young inequality with exponents $\alpha+1$ and $\frac{\alpha+1}{\alpha}$, we have
		\begin{align*}
			&\alpha C_1 \int_{\Sigma} (|\nabla u| + |\nabla v|)^{p-2} |\nabla w|^2 w^{\alpha-1} \varphi_R^{\alpha+1}\\
			&\le \int_{\Sigma} \left(\frac{[(\alpha+1)C_2(2C_0)^{p-1}]^{\alpha+1}}{(\alpha+1)(\frac{\alpha+1}{\alpha} C_\varepsilon)^\alpha} |\nabla \varphi_R|^{\alpha+1} + C_\varepsilon w^{\alpha+1} \varphi_R^{\alpha+1}\right) - C_\varepsilon \int_{\Sigma} w^{\alpha+1} \varphi_R^{\alpha+1}\\
			&\le C R^{N-\alpha-2}.
		\end{align*}
		
		Since $\alpha>N-2$, by letting $R\to+\infty$, we derive
		\[
		\int_{\Sigma} (|\nabla u| + |\nabla v|)^{p-2} |\nabla w|^2 w^{\alpha-1} = 0.
		\]
		This yields $u \le v + \varepsilon$ in ${\Sigma}$.
		
		Since $\varepsilon$ is arbitrary, we conclude that $u\le v$ in $\Sigma$.
	\end{proof}
	
	\begin{remark}\label{rem:ineq} It is clear from the proof that Proposition \ref{prop:wcp} still holds true if we replace the first equation of \eqref{ky} by
		\[
		-\Delta_p u \le f(u) \text{ in } \Sigma, \quad |\nabla u| \in L^\infty(\Sigma\cap\{u>\varepsilon\}) \text{ for all } \varepsilon>0,
		\]
		or replace the second equation of \eqref{ky} by
		\[
		-\Delta_p v \ge f(v) \text{ in } \Sigma, \quad |\nabla v| \in L^\infty(\Sigma).
		\]
	\end{remark}
	
	As an application of Proposition \ref{prop:wcp}, we prove Proposition \ref{prop:zero_convergence}, which provides a criterion for the uniform convergence of solutions to zero as $x_N\to 0^+$.	
	\begin{proof}[Proof of Proposition \ref{prop:zero_convergence}]
		Let $h:(0,+\infty) \to \mathbb{R}$ be a $C^1$ function such that
		\begin{align*}
			\max\{f(t), 0\} < h(t) &\quad\text{ in } (0,\rho),\\
			h(t) = \frac{c}{t^2} &\quad\text{ in } [\rho,+\infty)
		\end{align*}
		for some $c>0$.
		We set $H(t) = \int_{\rho}^{t} h(s) ds$ for $t\ge0$, then $H$ is strictly increasing in $(0,+\infty)$ and $H(t) < \int_{\rho}^{+\infty} h(s) ds = \frac{c}{\rho}$. For each $\mu \ge \frac{c}{\rho}$, we have
		\[
		\int_{0}^{+\infty} \frac{ds}{\left[\mu-H(s)\right]^\frac{1}{p}} = +\infty \text{ and } \int_{0}^{t} \frac{ds}{\left[\mu-H(s)\right]^\frac{1}{p}} < +\infty \text{ for } 0<t<+\infty,
		\]
		which is due to
		\[
		\int_{0}^{+\infty} \frac{ds}{\left[\mu-H(s)\right]^\frac{1}{p}} > \int_{\rho}^{+\infty} \frac{ds}{\mu^\frac{1}{p}} \text{ and }
		\int_{0}^{t} \frac{ds}{\left[\mu-H(s)\right]^\frac{1}{p}} < \frac{t}{\left[\mu-H(t)\right]^\frac{1}{p}}.
		\]
		Hence the formula
		\[
		\int_{0}^{w_\mu(t)} \frac{ds}{\left[\mu-H(s)\right]^\frac{1}{p}} = \left(\frac{p}{p-1}\right)^\frac{1}{p}t \quad\text{ for all } t\ge0
		\]
		uniquely define a function $w_\mu:[0,+\infty)\to\mathbb{R}$, which is a $C^2(\mathbb{R}_+) \cap C(\overline{\mathbb{R}_+})$ solution to the ODE problem
		\[
		\begin{cases}
			-(|w'|^{p-2}w')' = h(w) &\text{ in } \mathbb{R}_+,\\
			w(t) > 0, ~ w'(t) > 0 &\text{ in } \mathbb{R}_+,\\
			w(0) = 0.
		\end{cases}
		\]
		Moreover, $\lim_{\mu\to+\infty} w_\mu(t) = +\infty$ for all $t>0$.
		
		We fix some $\mu>0$ such that $w_\mu(\overline{\lambda}) > \rho$. Then we choose $\lambda_0<\overline\lambda$ satisfying $\|u\|_{L^\infty(\Sigma_{\overline\lambda})} < w_\mu(\lambda_0) < \rho$. By abuse of notation, we will write $w_\mu(x',x_N):=w_\mu(x_N)$. Then $0<w_\mu<\rho$ in $\Sigma_{\lambda_0}$ and $u < w_\mu$ on $\{x_N=\lambda_0\}$.
		
		For small $\varepsilon>0$ such that $w_\mu(\lambda_0+\varepsilon)<\rho$, we define
		\[
		w_{\mu,\varepsilon}(x) := w_\mu(x + \varepsilon e_N).
		\]
		Then
		\[
		\begin{cases}
			w_\mu(\varepsilon) < w_{\mu,\varepsilon} < w_\mu(\lambda_0+\varepsilon) < \rho & \text{ in } \Sigma_{\lambda_0},\\
			-\Delta_p w_{\mu,\varepsilon} = h(w_{\mu,\varepsilon}) > f(w_{\mu,\varepsilon}) & \text{ in } \Sigma_{\lambda_0},\\
			u \le w_{\mu,\varepsilon} & \text{ on } \partial\Sigma_{\lambda_0}.
		\end{cases}
		\]
		Now Proposition \ref{prop:wcp} implies $u \le w_{\mu,\varepsilon}$ in $\Sigma_{\lambda_0}$. Letting $\varepsilon\to0$, we have $u \le w_{\mu}$ in $\Sigma_{\lambda_0}$ and the conclusion follows from that fact that $\lim_{t\to0^+} w_{\mu}(t) = 0$.
	\end{proof}
	
	\subsection{A priori bounds for solutions}
	Motivated by \cite{MR4753083}, we prove some a priori bounds for solutions to \eqref{main}.	
	The following lemma provides an upper bound for solutions near the boundary.
	\begin{lemma}\label{lem:upper_bound}
		Let $f:(0,+\infty)\to\mathbb{R}$ be a locally Lipschitz continuous function such that $f(t) < \frac{c_0}{t^\gamma}$ for all $0<t<t_0$, where $c_0,t_0>0$, $\gamma>1$.
		Let $u \in C^{1,\alpha}_{\rm loc}(\mathbb{R}^N_+) \cap C(\overline{\mathbb{R}^N_+})$ be a solution to \eqref{main} with $\|u\|_{L^\infty(\Sigma_{\overline\lambda})} < +\infty$ for some $\overline\lambda>0$. Then
		\[
		u(x) \le C x_N^\frac{p}{\gamma+p-1} \quad\text{ in } \Sigma_{\overline\lambda}
		\]
		for some constants $C>0$. In particular, such a solution satisfies \eqref{bound_strips}.
	\end{lemma}
	\begin{proof}
		Setting
		\[
		M = \max\{\|u\|_{L^\infty(\Sigma_{\overline\lambda})}, t_0\},
		\]
		then, by continuity, there exists $c_1\ge c_0$ such that $f(t) < \frac{c_1}{t^\gamma}$ for all $0<t<M$.
		Let
		\[
		w(t) := \left[\frac{(\gamma+p-1)^p}{p^{p-1}(p-1)(\gamma-1)}\right]^\frac{1}{\gamma+p-1}	t^\frac{p}{\gamma+p-1},
		\]
		then $v_s(x) := s w(x_N)$ solves $-\Delta_p v_s = \frac{s^{\gamma+p-1}}{v_s^\gamma}$ in $\mathbb{R}^N_+$. We choose sufficiently large $s$ such that $s^{\gamma+p-1} \ge c_1$ and $s w(\overline\lambda) \ge M$.		
		For small $\varepsilon>0$, we define
		\[
		v_{s,\varepsilon}(x) := v_s(x + \varepsilon e_N) = s w(x_N + \varepsilon).
		\]
		
		We have
		\[
		\begin{cases}
			-\Delta_p u = f(u) < \frac{c_1}{u^\gamma} & \text{ in } \Sigma_{\overline\lambda},\\
			-\Delta_p v_{s,\varepsilon} = \frac{s^{\gamma+p-1}}{v_{s,\varepsilon}^\gamma} \ge \frac{c_1}{v_{s,\varepsilon}^\gamma} & \text{ in } \Sigma_{\overline\lambda},\\
			0 < u < M + 1 & \text{ in } \Sigma_{\overline\lambda},\\
			s w(\varepsilon) < v_{s,\varepsilon} < s w(\overline\lambda+\varepsilon) & \text{ in } \Sigma_{\overline\lambda},\\
			u \le v_s & \text{ on } \partial\Sigma_{\overline\lambda}.
		\end{cases}
		\]
		
		We know that $|\nabla u| \in L^\infty(\Sigma_{\overline\lambda}\cap\{u>\varepsilon\}) \text{ for all } \varepsilon>0$ by the standard regularity estimate.
		Therefore, Proposition \ref{prop:wcp} and Remark \ref{rem:ineq} imply $u \le v_{s,\varepsilon}$ in $\Sigma_{\overline\lambda}$. Letting $\varepsilon\to0$, we conclude the proof.
	\end{proof}
	
	The following lemma provides a lower bound for solutions.
	\begin{lemma}\label{lem:lower_bound}
		Let $f:(0,+\infty)\to\mathbb{R}$ be a locally Lipschitz continuous function such that $f(t) > \frac{c_0}{t^\gamma}$ for all $0<t<t_0$, where $c_0,t_0>0$, $\gamma\ge0$. Let $u \in C^{1,\alpha}_{\rm loc}(\mathbb{R}^N_+) \cap C(\overline{\mathbb{R}^N_+})$ be a solution to \eqref{main}. Then
		\[
		u(x) \ge \min\{C x_N^\frac{p}{\gamma+p-1}, t_0\} \quad\text{ in } \mathbb{R}^N_+.
		\]
		for some constant $C>0$.
	\end{lemma}
	
	\begin{proof}
		Let $\lambda_1>0$ and $\phi_1 \in C^1(\overline{B_1})$ be the first eigenvalue and a corresponding positive eigenfunction of the $p$-Laplacian in $B_1$, namely,
		\[
		\begin{cases}
			-\Delta_p \phi_1 = \lambda_1 \phi_1^{p-1} & \text{ in } B_1,\\
			\phi_1 > 0 & \text{ in } B_1,\\
			\phi_1 = 0 & \text{ on } \partial B_1.
		\end{cases}
		\]
		
		Setting
		\[
		w = s \phi_1^\frac{p}{\gamma+p-1},
		\]
		where $s>0$ will be chosen later. Direct calculation yields that in the weak sense
		\[
		-\Delta_p w = \frac{\alpha(x)}{w^\gamma} \quad \text{ in } B_1,
		\]
		where
		\[
		\alpha(x) := s^{\gamma+p-1} \left(\frac{p}{\gamma+p-1}\right)^{p-1} \left[\frac{(\gamma-1)(p-1)}{\gamma+p-1}|\nabla\phi_1(x)|^p + \lambda_1\phi_1(x)^p\right].
		\]
		
		Now we fix $s>0$ such that $\sup_{x\in B_1}\alpha(x) \le c_0$ and hence
		\[
		-\Delta_p w \le \frac{c_0}{w^\gamma} \text{ in } B_1.
		\]
		Let $R_0>0$ be such that $R_0^\frac{p}{\gamma+p-1} w(0) = t_0$.
		
		For any $0<R\le R_0$ and $x_0 = (x_0', x_{0,N})\in\mathbb{R}^N_+$ with $x_{0,N} \ge R + \varepsilon$, where $\varepsilon$ is sufficiently small, we set
		\[
		w_{x_0,R}(x) := R^\frac{p}{\gamma+p-1}w\left(\frac{x-x_0}{R}\right) \quad\text{ in } B_R(x_0).
		\]
		Then
		\[
		w_{x_0,R} \le t_0 \quad\text{ and }\quad -\Delta_p w_{x_0,R} \le \frac{c_0}{w_{x_0,R}^\gamma} \quad\text{ in } B_R(x_0).
		\]
		
		On the other hand, since $w_{x_0,R}=0 < u$ on $\partial B_R(x_0)$, we can use $(w_{x_0,R}-u)^+ \chi_{B_R(x_0)}$ as a test function in
		\[
		-\Delta_p u = f(u)
		\quad\text{ and }\quad
		-\Delta_p w_{x_0,R} \le \frac{c_0}{w_{x_0,R}^\gamma}
		\]
		to obtain
		\begin{align*}
			&\int_{B_R(x_0)} (|\nabla w_{x_0,R}|^{p-2}\nabla w_{x_0,R} - |\nabla u|^{p-2}\nabla u, \nabla (w_{x_0,R}-u)^+) \\
			&\le \int_{B_R(x_0)} \left(\frac{c_0}{w_{x_0,R}^\gamma} - f(u)\right) (w_{x_0,R}-u)^+.
		\end{align*}
		
		In $B_R(x_0)\cap\{w_{x_0,R} > u\}$ we have $f(u) \ge \frac{c_0}{u^\gamma}$. Hence
		\begin{align*}
			&\int_{B_R(x_0)} (|\nabla w_{x_0,R}|^{p-2}\nabla w_{x_0,R} - |\nabla u|^{p-2}\nabla u, \nabla (w_{x_0,R}-u)^+) \\
			&\le \int_{B_R(x_0)} \left(\frac{c_0}{w_{x_0,R}^\gamma} - \frac{c_0}{u^\gamma}\right) (w_{x_0,R}-u)^+ \le 0.
		\end{align*}
		
		By \eqref{p_ge}, this implies
		\[
		\int_{B_R(x_0)} (|\nabla w_{x_0,R}| + |\nabla u|)^{p-2} |\nabla (w_{x_0,R}-u)^+|^2 \le 0.
		\]
		Hence $u \ge w_{x_0,R}$ in $B_R(x_0)$ with $x_{0,N} \ge R + \varepsilon$. Since $\varepsilon>0$ is arbitrary, we deduce
		\[
		u \ge w_{x_0,R} \text{ in } B_R(x_0) \text{ for all } 0<R\le R_0 \text{ and } x_0 \in \mathbb{R}^N_+ \text{ with } x_{0,N} \ge R.
		\]
		
		In particular, if $x_{0,N} = R < R_0$, then
		\[
		u(x_0) \ge w_{x_0,R}(x_0) = w(0)R^\frac{p}{\gamma+p-1} = w(0) x_{0,N}^\frac{p}{\gamma+p-1}.
		\]
		If $x_{0,N} \ge R = R_0$, then
		\[
		u(x_0) \ge w_{x_0,R}(x_0) = w(0)R_0^\frac{p}{\gamma+p-1} = t_0.
		\]
		
		The conclusion follows from the fact that $x_0$ is chosen arbitrarily in $\mathbb{R}^N_+$.
	\end{proof}
	We still have a lower bound under weaker assumptions on $f$.
	\begin{lemma}\label{lem:positive}
		Let $f:(0,+\infty)\to\mathbb{R}$ be a locally Lipschitz continuous function such that $f(t) > c_0t^{p-1}$ for all $0<t<t_0$, where $c_0,t_0>0$. Let $u \in C^{1,\alpha}_{\rm loc}(\mathbb{R}^N_+) \cap C(\overline{\mathbb{R}^N_+})$ be a solution to \eqref{main}.
		Then
		\[
		u(x) \ge \min\{C x_N, t_0\} \quad\text{ in } \mathbb{R}^N_+.
		\]
		for some constant $C>0$.
	\end{lemma}	
	A weaker result was proved in \cite{10.1007/s11118-024-10157-1} exploiting the weak sweeping principle (see also \cite[Lemma 3]{MR3118616} for the case that $p>2$ and $f$ is positive).
	More precisely, Lemma 10 in \cite{10.1007/s11118-024-10157-1} is stated for nonlinearity $f$ that is continuous at zero and the conclusion there does not provide an explicit lower bound for $u$. To get a stronger result, we still use the weak sweeping principle, but in a different way.
	\begin{proof}[Proof of Lemma \ref{lem:positive}]
		Let $\lambda_1>0$ and $\phi_1 \in C^1(\overline{B_1})$ be the first eigenvalue and the corresponding positive eigenfunction of the $p$-Laplacian in $B_1$ such that $\phi_1(0)=t_0$. We take $R=\sqrt{\frac{2\lambda_1}{c_0}}$ and set $\phi_R(x) = \phi_1\left(\frac{x}{R}\right)$, then
		\[
		\begin{cases}
			-\Delta_p \phi_R = \frac{c_0}{2} \phi_R^{p-1} & \text{ in } B_R,\\
			\phi_R > 0 & \text{ in } B_R,\\
			\phi_R = 0 & \text{ on } \partial B_R.
		\end{cases}
		\]
		Since $\phi_R$ is radially symmetric and by abuse of notation, we may write $\phi_R(x)=\phi_R(|x|)$.
		For each $x_0\in\mathbb{R}^N_+\setminus\Sigma_R$ we set
		\[
		\phi_R^{x_0}(x) = \phi_R(x - x_0) \quad\text{ for } x \in B_R(x_0).
		\]
		We will show that
		\begin{equation}\label{ky2}
			u \ge \phi_R^{x_0} \text{ in } B_R(x_0) \quad\text{ for every } x_0\in\mathbb{R}^N_+\setminus\Sigma_R.
		\end{equation}
		To this end, we let any $x_0:=(x_0', x_{0,N})\in\mathbb{R}^N_+\setminus\Sigma_R$.
		
		We only consider the case $x_{0,N}>R$ since the case $x_{0,N}=R$ can be obtained by continuity. Let $s_0\in(0,1)$ be such that $\delta:=\min_{\overline{B_R(x_0)}} u > s_0\phi_R^{x_0}$ in $B_R(x_0)$ and let $\varepsilon>0$ be such that $\phi_R(R-\varepsilon) < \frac{\delta}{2}$. We denote $\tilde{\phi}_s=s\phi_R^{x_0}$. Then for all $s \in [s_0, 1]$, we have
		\[
		\begin{cases}
			-\Delta_p u = f(u) & \text{ in } B_{R-\varepsilon},\\
			-\Delta_p \tilde{\phi}_s = \frac{c_0}{2} \tilde{\phi}_s^{p-1} \le f(s\phi_R^{x_0}) - \gamma & \text{ in } B_{R-\varepsilon},\\
			u \ge \tilde{\phi}_s + \frac{\delta}{2} & \text{ on } \partial B_{R-\varepsilon},
		\end{cases}
		\]
		where
		\[
		\gamma=\frac{1}{2}\min_{[s_0\phi_R(R-\varepsilon),t_0]} f>0.
		\]
		Moreover, $u > \tilde{\phi}_{s_0}$ in $B_{R-\varepsilon}$. Thus we can apply the weak sweeping principle (Theorem \ref{th:wsp}) to deduce that $u \ge \tilde{\phi}_s$ in $B_{R-\varepsilon}$ for all $s \in [s_0, 1]$. In particular, $u \ge \tilde{\phi}_1 = \phi_R^{x_0}$ in $B_{R-\varepsilon}$. Since $\varepsilon$ is arbitrary, \eqref{ky2} must hold. This implies
		\[
		u(x) \ge \begin{cases}
			\phi_R(R-x_N) &\text{ if } x_N< R,\\
			\phi_R(0) &\text{ if } x_N\ge R.
		\end{cases}
		\]
		The conclusion follows immediately from the fact that $\phi_R'(R)<0$ and $\phi_R(0)=t_0$.
	\end{proof}
	
	\section{Monotonicity of solutions}\label{sect3}
	
	For $\lambda>0$, we define
	\[
	u_\lambda (x', x_N) := u(x',2\lambda-x_N),
	\]
	which is obtained by reflecting $u$ with respect to the hyperplane $T_\lambda := \{(x',x_N)\in\mathbb{R}^N \mid x_N=\lambda\}$. The following proposition allows us to initiate the moving plane procedure.
	
	\begin{proposition}\label{prop:mn} Under the assumptions of Theorem \ref{th:monotonicity}, we have
		\[
		u\leq u_\lambda \text{ in } \Sigma_\lambda \quad\text{ for all } 0<\lambda\leq\overline\lambda,
		\]
		where $\overline\lambda>0$.
	\end{proposition}
	
	\begin{proof}
		Let $\overline\lambda>0$ be such that $\|u\|_{L^\infty(\Sigma_{\overline\lambda})} < t_0$.
		Using Lemma \ref{lem:lower_bound} with $\gamma=0$, one may check that for $0<\lambda\leq\overline\lambda$,
		\[
		\begin{cases}
			-\Delta_p u = f(u) &\text{ in } \Sigma_\lambda,\\
			-\Delta_p u_\lambda = f(u_\lambda) &\text{ in } \Sigma_\lambda,\\
			0<u<t_0 &\text{ in } \Sigma_\lambda,\\
			\min\{C \lambda^\frac{p}{p-1}, t_1\} \le u_\lambda\le\|u\|_{L^\infty(\Sigma_{2\lambda})} &\text{ in } \Sigma_\lambda,\\
			u \le u_\lambda &\text{ on } \partial\Sigma_\lambda,
		\end{cases}
		\]
		where $t_1>0$.
		Hence, Proposition \ref{prop:wcp} yields $u\leq u_\lambda$ in $\Sigma_\lambda$ for all $0<\lambda\leq\overline\lambda$.
	\end{proof}
	
	\begin{proof}[Proof of Theorem \ref{th:monotonicity}]	
		Due to Proposition \ref{prop:mn}, the set
		\[
		\Lambda:=\left\{\lambda>0 \mid u\leq u_\mu \text{ in } \Sigma_\mu \text{ for all } 0<\mu\leq \lambda\right\}
		\]
		is nonempty. Thus, we can define
		\begin{equation}\label{lambda}
			\lambda_0=\sup\Lambda.
		\end{equation}
		To obtain the monotonicity of $u$, it suffices to show that $\lambda_0=+\infty$. By contradiction arguments, we assume $\lambda_0<+\infty$. Then $u\leq u_{\lambda_0}$ in $\Sigma_{\lambda_0}$. We can reach a contradiction by showing that for some small $\varepsilon>0$ we have
		\[
		u\leq u_\lambda \text{ in } \Sigma_\lambda \quad\text{ for all } \lambda_0 < \lambda<\lambda_0 + \varepsilon.
		\]
		Due to Lemma \ref{lem:lower_bound} (with $\gamma=0$), there exist $\tilde{\lambda}, \tilde{\delta}>0$ small such that
		\[
		u+\tilde{\delta}\leq u_\lambda \text{ in } \Sigma_{\tilde{\lambda}} \quad\text{ for all } \lambda>\lambda_0.
		\]
		Therefore, we only need to show that
		\begin{equation}\label{gf}
			u\leq u_\lambda \text{ in } \Sigma_\lambda\setminus\Sigma_{\tilde{\lambda}} \quad\text{ for all } \lambda_0 < \lambda<\lambda_0 + \varepsilon
		\end{equation}
		for some $\varepsilon\in(0,1)$.
		By Lemma \ref{lem:lower_bound} again, we know that
		\[
		\min\{C \tilde{\lambda}^\frac{p}{p-1}, t_1\} \le u, u_\lambda \le \|u\|_{L^\infty(\Sigma_{2\lambda})} \quad\text{ in } \Sigma_\lambda\setminus\Sigma_{\tilde{\lambda}}.
		\]
		Hence $f(u)$ and $f(u_\lambda)$ are bounded in $\Sigma_\lambda\setminus\Sigma_{\tilde{\lambda}}$. Therefore, by standard gradient elliptic estimates, we have $|\nabla u|, |\nabla u_\lambda| \in L^\infty(\Sigma_\lambda\setminus\Sigma_{\tilde{\lambda}})$ for every $\lambda>0$. Hence, we can repeat the techniques in \cite{MR3303939,MR2886112,MR3752525,MR3118616,MR4439897}, which are based on various comparison principles and compactness arguments for problems with a regular nonlinearity, to prove \eqref{gf}. More precisely, if $f$ is positive and $1<p<2$, we use the arguments in \cite{MR3303939}. If $f$ is positive and $p>2$, we follow the ones in \cite{MR3118616}. When $f$ is sign-changing and \( \frac{2N+2}{N+2}<p<2 \), we argue as in \cite{MR4439897} (see also \cite{MR4750390,10.1007/s11118-024-10157-1} for simplified arguments).
		
		The details, therefore, will be omitted.
	\end{proof}
	
	Next, we prove Theorem \ref{th:monotonicity2}. Our proof is motivated by the scaling technique in \cite{MR3459013}.
	\begin{proof}[Proof of Theorem \ref{th:monotonicity2}]
		Since $g:[0,+\infty)$ is a locally Lipschitz continuous, there exist $t_0,c_1,c_2>0$ such that
		\[
		\frac{c_1}{t^\gamma}<f(t)<\frac{c_2}{t^\gamma}\quad\text{ in } (0,t_0).
		\]
		Hence Lemmas \ref{lem:upper_bound} and \ref{lem:lower_bound} imply the existence of $\lambda_0,c,C>0$ such that
		\begin{equation}\label{bounds}
			c x_N^\frac{p}{\gamma+p-1} \le u(x) \le C x_N^\frac{p}{\gamma+p-1} \quad\text{ in } \Sigma_{\lambda_0}.
		\end{equation} 
		
		Let any $A>a>0$ and any positive sequence $(\varepsilon_n)$ such that $\varepsilon_n\to0$ as $n\to\infty$. We define
		\[
		w_n(x) := \varepsilon_n^{-\frac{p}{\gamma+p-1}} u(\varepsilon_n x) \quad\text{ for } x \in \mathbb{R}^N_+.
		\]
		For $n$ sufficiently large, we deduce from \eqref{bounds}
		\begin{equation}\label{wn_bound}
			c a^\frac{p}{\gamma+p-1} \le c x_N^\frac{p}{\gamma+p-1} \le w_n(x) \le C x_N^\frac{p}{\gamma+p-1} \le C A^\frac{p}{\gamma+p-1} \quad\text{ in } \Sigma_A\setminus \Sigma_a.
		\end{equation}
		In particular, $(w_n)$ is uniformly bounded in $L^\infty(\Sigma_A\setminus \Sigma_a)$ and it solves
		\begin{equation}\label{wn}
			-\Delta_p w_n = \frac{1}{w_n^\gamma} + \varepsilon_n^\frac{\gamma p}{\gamma+p-1} g(\varepsilon_n^\frac{p}{\gamma+p-1} w_n) \quad \text{ in } \mathbb{R}^N_+.
		\end{equation}
		By the standard regularity \cite{MR1814364}, $(w_n)$ is also uniformly bounded in $L^\infty(\Sigma_A\setminus \Sigma_a)$ and in $C^{1,\alpha}(\overline{\Sigma_A\setminus \Sigma_a})$, for $0 < \alpha < 1$. Since
		\[
		|\nabla w_n(x)| = \varepsilon_n^{\frac{\gamma-1}{\gamma+p-1}} |\nabla u(\varepsilon_n x)| \ge \varepsilon_n^{\frac{\gamma-1}{\gamma+p-1}} \frac{\partial u(\varepsilon_n x)}{\partial \eta},
		\]
		for $\varepsilon_n$ sufficiently small, we get the
		estimate from above in \eqref{gradient_blowup}.
		
		Now we prove the estimate from below. Suppose by contradiction that there exist $\beta>0$, a sequence of normal vectors $\eta_n\in \mathbb{S}^{N-1}_+$ with $(\eta_n,e_N) \ge \beta$ and a sequence of points $x_n=(x_n', x_{n,N})\in\mathbb{R}^N_+$ such that
		\begin{equation}\label{cassump}
			x_{n,N}^{\frac{\gamma-1}{\gamma+p-1}} \frac{\partial u(x_n)}{\partial \eta_n} \to 0 \text{ and } x_{n,N} \to 0 \quad\text{ as } n\to\infty.
		\end{equation}
		Passing to a subsequence, we may assume $\eta_n\to\eta\in \mathbb{S}^{N-1}_+$ with $(\eta,e_N) \ge \beta$ as $n\to\infty$.
		We define $w_n$ as above with $\varepsilon_n=x_{n,N}$ and $\tilde w_n(x',x_N)=w_n(x'+\varepsilon_n^{-1}x'_n,x_N)$, namely,
		\[
		\tilde w_n(x) := x_{n,N}^{-\frac{p}{\gamma+p-1}} u(x_{n,N} x'+x_n', x_{n,N} x_N) \quad\text{ for } x \in \mathbb{R}^N_+.
		\]
		Then \eqref{wn_bound} and \eqref{wn} still hold for $\tilde w_n$.
		Moreover, $(\tilde w_n)$ is uniformly bounded in $C^{1,\alpha}(\overline{\Sigma_A\setminus \Sigma_a})$. Hence, up to a subsequence, we have
		\[
		\tilde w_n \to w_{a,A} \quad\text{ in } C^{1,\alpha'}_{\rm loc}(\overline{\Sigma_A\setminus \Sigma_a}),
		\]
		where $0<\alpha'<\alpha$.
		Moreover, passing \eqref{wn} to the limit, we get
		\[
		-\Delta w_{a,A} = \frac{1}{w_{a,A}^\gamma} \quad\text{ in } \Sigma_A\setminus \Sigma_a.
		\]
		Now we take $a = \frac{1}{j}$ and $A = j$, for large $j \in \mathbb{N}$ and we construct $w_{\frac{1}{j},j}$ as above. For $j\to\infty$, using a standard diagonal process, we can construct a limiting profile $w_\infty \in C^{1,\alpha'}_{\rm loc}(\mathbb{R}^N_+)$ so that
		\[
		-\Delta w_\infty = \frac{1}{w_\infty^\gamma} \quad\text{ in } \mathbb{R}^N_+
		\]
		and $w_{\frac{1}{j},j} = w_\infty$ in $\Sigma_j\setminus \Sigma_{\frac{1}{j}}$. Moreover, from \eqref{wn_bound} we know that
		\[
		c x_N^\frac{p}{\gamma+p-1} \le w_\infty(x) \le C x_N^\frac{p}{\gamma+p-1} \quad\text{ in } \mathbb{R}^N_+.
		\]
		Hence by defining $w_\infty=0$ on $\partial\mathbb{R}^N_+$, we have $w_\infty \in C^{1,\alpha'}_{\rm loc}(\mathbb{R}^N_+)\cap C(\overline{\mathbb{R}^N_+})$ and $w_\infty$ is a solution to \eqref{pure}.
		By \cite[Theorem 1.2]{MR4044739},
		\[
		w_\infty(x) \equiv \left[\frac{(\gamma+p-1)^p}{p^{p-1}(p-1)(\gamma-1)}\right]^\frac{1}{\gamma+p-1}	x_N^\frac{p}{\gamma+p-1}.
		\]
		
		On the other hand, \eqref{cassump} gives $\frac{\partial \tilde w_n(e_N)}{\partial \eta_n} = x_{n,N}^\frac{\gamma-1}{\gamma+p+1} \frac{\partial u(x_n)}{\partial \eta_n} \to 0$ as $n \to \infty$. This is a contradiction since $\frac{\partial \tilde w_n(e_N)}{\partial \eta_n} \to \frac{\partial w_\infty(e_N)}{\partial \eta}=w_\infty'(1)\eta_N>0$.
		
		Hence \eqref{gradient_blowup} is proved. In particular,
		\[
		u\leq u_\lambda \text{ in } \Sigma_\lambda \quad\text{ for all } 0<\lambda\leq\frac{\lambda_0}{2}.
		\]
		From this, we can proceed as in the proof of Theorem \ref{th:monotonicity} to deduce the monotonicity of $u$ in $\mathbb{R}^N_+$.
	\end{proof}
	
	\section{1D symmetry of solutions}\label{sect4}
	
	\subsection{Weak comparison principles for half-spaces}
	We start this section with the following comparison principles for half-spaces.
	\begin{proposition}\label{prop:wcp2}
		Let $f:(0,+\infty)\to\mathbb{R}$ be a locally Lipschitz continuous function such that $f$ is strictly decreasing on $(t_0,+\infty)$ for some $t_0>0$ and let $\Sigma := \mathbb{R}^N_+\setminus \overline\Sigma_\lambda$ for some $\lambda>0$. Assume that $u, v \in C^{1,\alpha}_{\rm loc}(\Sigma) \cap C(\overline{\Sigma})$ satisfy
		\[
		\begin{cases}
			-\Delta_p u = f(u) &\text{ in } \Sigma,\\
			-\Delta_p v = f(v) &\text{ in } \Sigma,\\
			u>0 &\text{ in } \Sigma,\\
			v>t_0 &\text{ in } \Sigma,\\
			u \le v &\text{ on } \partial\Sigma,
		\end{cases}
		\]
		and
		\[
		u\in L^\infty(\Sigma\cap\Sigma_\mu)\text{ for all } \mu>\lambda,
		\]
		\begin{equation}\label{lim}
			\lim_{x_N\to+\infty} (u-v) = 0 \text{ uniformly in } x'\in\mathbb{R}^{N-1}.
		\end{equation}
		Then $u\le v$ in $\Sigma$.
	\end{proposition}
	\begin{proof}
		Fix some $\alpha>N-2$ and $\varepsilon>0$. Then we set $w=(u-v-\varepsilon)^+$ and
		\begin{equation}\label{psi}
			\psi(x) := w^{\alpha}(x) \varphi_R^{\alpha+1}(x') \chi_{\Sigma}(x),
		\end{equation}
		where $\varphi_R$ is defined as in \eqref{phi}.
		Using \eqref{lim}, we find that $w=0$ if $x_N>M$ for some $M>\lambda$ independent of $R$. Hence, the support of $\psi$ is compactly contained in $\Sigma$, and we can use $\psi$ as a test function in the equations $-\Delta_p u = f(u)$ and $-\Delta_p v = f(v)$. We can proceed as the proof of Proposition \ref{prop:wcp} until we reach \eqref{t0}.
		
		In the set ${\Sigma}\cap\{w > 0\}$, we have
		\begin{equation}\label{bound}
			t_0 < v < v+\varepsilon < u \le \|u\|_{L^\infty(\Sigma\cap \Sigma_M)}.
		\end{equation}
		Since $f$ is strictly decreasing on $(t_0,+\infty)$, Lemma \ref{lem:k} gives
		\begin{equation}\label{t2'}
			f(u) - f(v) \le - C_\varepsilon (u-v) \le - C_\varepsilon w \quad\text{ in } {\Sigma}\cap\{w > 0\}
		\end{equation}
		for some $C_\varepsilon>0$.
		On the other hand, from \eqref{bound} and the fact that $f$ is locally Lipschitz continuous in $(0,+\infty)$, the standard gradient estimate yields
		\begin{equation}\label{t3'}
			|\nabla u| < C_0 \quad\text{ and }\quad |\nabla v| < C_0 \quad\text{ in } {\Sigma}\cap\{w > 0\}.
		\end{equation}
		Now we can plug \eqref{t2'} and \eqref{t3'} into \eqref{t0} and proceed as in the proof of Proposition \ref{prop:wcp} until we finish the proof.
	\end{proof}
	
	\begin{proposition}\label{prop:wcp3}
		Let $f:(0,+\infty)\to\mathbb{R}$ be a locally Lipschitz continuous function such that $f$ is differentiable and strictly decreasing on $(t_0,+\infty)$ for some $t_0>0$ and
		\[
		\limsup_{t\to+\infty} f'(t) < 0.
		\]		
		Let $\Sigma := \mathbb{R}^N_+\setminus \overline\Sigma_\lambda$ for some $\lambda>0$. Assume that $u, v \in C^{1,\alpha}_{\rm loc}(\Sigma) \cap C(\overline{\Sigma})$ satisfy
		\[
		\begin{cases}
			-\Delta_p u \le f(u) &\text{ in } \Sigma,\\
			-\Delta_p v \ge f(v) &\text{ in } \Sigma,\\
			u>0 &\text{ in } \Sigma,\\
			v>t_0 &\text{ in } \Sigma,\\
			u \le v &\text{ on } \partial\Sigma
		\end{cases}
		\]
		and
		\[
		|\nabla u|, |\nabla v| \in L^\infty(\Sigma).
		\]
		Then $u\le v$ in $\Sigma$.
	\end{proposition}
	\begin{proof}
		The proof is similar to that of Proposition \ref{prop:wcp2}. However, the support of function $\psi$ defined as in \eqref{psi} may be unbounded. Instead, we will define $\psi$ as
		\[	
		\psi(x) := w^{\alpha}(x) \phi_R^{\alpha+1}(x') \chi_{\Sigma}(x),
		\]
		where $\alpha>N-1$ and $\phi_R \in C^1(\mathbb{R}^N)$ is a standard cutoff function such that
		\begin{equation}
			\begin{cases}
				0 \le \phi_R \le 1 &\text{ in } \mathbb{R}^N,\\
				\phi_R = 1 &\text{ in } B_R,\\
				\phi_R = 0 &\text{ in } \mathbb{R}^N \setminus B_{2R},\\
				|\nabla\phi_R| \le \frac{2}{R} &\text{ in } B_{2R} \setminus B_R.
			\end{cases}
		\end{equation}
		With this new choice of test function, we can proceed as in the proofs of Propositions \ref{prop:wcp2} and \ref{prop:wcp3}. Notice that in our situation, \eqref{bound} is replaced with
		\[
		t_0 < v < v+\varepsilon < u
		\]
		and \eqref{t2'} still holds thanks to Lemma \ref{lem:k}.
	\end{proof}
	
	\subsection{Positive nonlinearity}
	In this subsection, we consider the case that $f$ is positive and $p>\frac{2N+2}{N+2}$.
	\begin{proof}[Proof of Theorem \ref{th:rigidity1}]	
		Let $\nu\in \mathbb{S}^{N-1}_+$.
		For each $\lambda>0$, we define
		\[
		u_\lambda^\nu (x) := u(x+\lambda\nu).
		\]
		
		We aim to show that
		\begin{equation}\label{nu_monotone}
			u \le u_\lambda^\nu \,\text{ in } \mathbb{R}^N_+ \quad\text{ for all } \lambda>0.
		\end{equation}
		
		From (i) and (ii), there exists $c_0>0$ such that $f(t)>c_0$ for $t>t_0+1$. By Lemma \ref{lem:positive}, there exists $\lambda^*>0$ such that
		\begin{equation}\label{b}
			u(x) \ge t_0+1 \text{ for } x_N>\lambda^*.
		\end{equation}
		Hence $u_\lambda^\nu \ge t_0+1$ in $\mathbb{R}^N_+$ for all $\lambda>\lambda_\nu^*$, where $\lambda_\nu^*:=\frac{\lambda^*}{(\nu,e_N)}$. Moreover, from \eqref{uni_gradient_zero} and the mean value theorem, we deduce
		\[
		\lim_{x_N\to+\infty} (u-u_\lambda^\nu) = 0 \text{ uniformly in } x'\in\mathbb{R}^{N-1}.
		\]
		
		Let any $\lambda>\lambda_\nu^*$. Since
		\[
		\begin{cases}
			-\Delta_p u = f(u) &\text{ in } \mathbb{R}^N_+,\\
			-\Delta_p u_\lambda^\nu = f(u_\lambda^\nu) &\text{ in } \mathbb{R}^N_+,\\
			u>0 &\text{ in } \mathbb{R}^N_+,\\
			u_\lambda^\nu>t_0 &\text{ in } \mathbb{R}^N_+,\\
			u \le u_\lambda^\nu &\text{ on } \partial\mathbb{R}^N_+,
		\end{cases}
		\]
		we can apply Proposition \ref{prop:wcp2} to derive
		\begin{equation}\label{reflection_comparison2}
			u \le u_\lambda^\nu \,\text{ in } \mathbb{R}^N_+ \quad\text{ for all } \lambda>\lambda_\nu^*.
		\end{equation}
		
		Now that the set
		\[
		\Lambda = \{ \lambda>0 \mid u \le u_\mu^\nu \text{ in } \mathbb{R}^N_+ \text{ for all } \mu > \lambda \}
		\]
		is nonempty, we can define
		\[
		\lambda_0 = \inf \Lambda.
		\]
		
		We will show that
		\[
		\lambda_0 = 0.
		\]
		
		Assume, on the contrary, that $\lambda_0>0$. By continuity of $u$, we have $u \le u_{\lambda_0}^\nu$ in $\mathbb{R}^N_+$. To reach a contradiction, we will search for some $\varepsilon_0$ small such that
		\begin{equation}\label{key}
			u \le u_\lambda^\nu \quad\text{ in } \mathbb{R}^N_+
		\end{equation}
		for all $\lambda \in (\lambda_0-\varepsilon_0, \lambda_0)$.
		
		$\circ$ Due to Lemma \ref{lem:positive}, there exist $\tilde{\lambda},\tilde{\delta}>0$ sufficiently small such that
		\begin{equation}\label{reflection_comparison_in2x}
			u+\tilde{\delta}\leq u_\lambda^\nu \quad\text{ in } \Sigma_{\tilde{\lambda}}
		\end{equation}
		for all $\lambda>\lambda_0/2$.
		
		$\circ$ We claim that
		\begin{equation}\label{Sigma_b}
			u \le u_\lambda^\nu \quad\text{ in } \Sigma
		\end{equation}
		for all $\lambda \in (\lambda_0-\varepsilon_0, \lambda_0)$, where $\varepsilon_0>0$ is sufficiently small and
		\[
		\Sigma = \{x\in\mathbb{R}^N_+ \mid \tilde{\lambda}\ \le x_N\le \lambda^*\}.
		\]
		
		Assume that \eqref{Sigma_b} does not hold. Then there exist two sequences $\lambda_n \nearrow \lambda_0$ and $x_n := (x'_n, (x_n)_N) \in \mathbb{R}^{N-1} \times [\tilde{\lambda}, \lambda^*]$ such that
		\begin{equation}\label{contra_sequence}
			u(x_n) > u_{\lambda_n}^\nu(x_n).
		\end{equation}
		Moreover, we may assume $(x_n)_N \to y_0 \in [\tilde{\lambda}, \lambda^*]$. Now we set
		\[
		v_n(x', x_N) = u(x'+x'_n, x_N).
		\]
		
		Since $\min\{C \tilde{\lambda}^\frac{p}{\gamma+p-1}, t_0+1\} \le v_n \le \|u\|_{L^\infty(\Sigma_\lambda)}$ in $\Sigma_\lambda\setminus\Sigma_{\tilde{\lambda}}$, we have that $f(v_n)$ is bounded in $\Sigma_\lambda\setminus\Sigma_{\tilde{\lambda}}$ for each $\lambda>\tilde{\lambda}$. The standard regularity gives $\|v_n\|_{C^{1,\alpha}(\Sigma_\lambda\setminus\Sigma_{\tilde{\lambda}})} < C_\lambda$. By the Arzel\`a--Ascoli theorem, via a standard diagonal process, we have
		\[
		v_n \to v \quad\text{ in } C^{1,\alpha'}_{\rm loc}(\mathbb{R}^N_+\setminus\Sigma_{\tilde{\lambda}})
		\]
		up to a subsequence, for $0<\alpha'<\alpha$. Moreover, $v$ weakly solves $-\Delta_p v = f(v)$ in $\mathbb{R}^N_+\setminus\Sigma_{\tilde{\lambda}}$.
		Using the definition of $\lambda_0$ and passing \eqref{contra_sequence} to the limit, we have
		\begin{align*}
			&v \le v_{\lambda_0}^\nu \quad\text{ in } \mathbb{R}^N_+\setminus\Sigma_{\tilde{\lambda}},\\
			&v(x_0) = v_{\lambda_0}^\nu(x_0),
		\end{align*}
		where $x_0 = (0',y_0)$. On the other hand, by \eqref{reflection_comparison_in2x} we have $v + \tilde{\delta} \le v_{\lambda_0}^\nu$ on $\partial(\mathbb{R}^N_+\setminus\Sigma_{\tilde{\lambda}})$. Hence the strong comparison principle (Theorem \ref{th:scp}) implies $v < v_{\lambda_0}^\nu$ in $\mathbb{R}^N_+\setminus\Sigma_{\tilde{\lambda}}$. This contradicts the fact that $v(x_0) = v_{\lambda_0}^\nu(x_0)$.		
		Therefore, \eqref{Sigma_b} must hold.
		
		$\circ$ Next, we show that
		\begin{equation}\label{C_Sigma_b}
			u \le u_\lambda^\nu \quad\text{ in } \mathbb{R}^N_+\setminus\Sigma_{\lambda^*}
		\end{equation}
		for all $\lambda \in (\lambda_0-\varepsilon_0, \lambda_0)$.
		
		From \eqref{Sigma_b} and the continuity, we already have $u \le u_\lambda^\nu$ on $\partial (\mathbb{R}^N_+\setminus\Sigma_{\lambda^*})$. Moreover, $u_\lambda^\nu(x) \ge t_0+1$ for each $x\in\mathbb{R}^N_+\setminus\Sigma_{\lambda^*}$. Hence \eqref{C_Sigma_b} follows by applying Proposition \ref{prop:wcp2} with $u$ and $v:=u_\lambda^\nu$ on $\mathbb{R}^N_+\setminus\Sigma_{\lambda^*}$.
		
		Combining \eqref{reflection_comparison_in2x}, \eqref{Sigma_b} and \eqref{C_Sigma_b}, we obtain \eqref{key}. This contradicts the definition of $\lambda_0$ and hence \eqref{nu_monotone} is proved.
		
		Therefore, $u$ is monotone increasing in direction $\nu$ for all $\nu\in \mathbb{S}^{N-1}_+$. That is,
		\[
		\frac{\partial u}{\partial \nu} := (\nabla u, \nu) \ge 0 \quad \text{ in } \mathbb{R}^N_+.
		\]
		
		To deduce the 1D symmetry of $u$, we take $\zeta$ be any direction in $\{x\in \partial B_1 \mid x_N=0\}$. Let $\nu_n\in \mathbb{S}^{N-1}_+$ be a sequence converging to $\zeta$, we have $\frac{\partial u}{\partial \nu_n} \ge 0$. By sending $n\to\infty$, we deduce
		\[
		\frac{\partial u}{\partial\zeta} \ge 0 \quad \text{ in } \mathbb{R}^N_+.
		\]
		Similarly, let another sequence $\tau_n\in \mathbb{S}^{N-1}_+$ converging to $-\zeta$, we obtain
		\[
		\frac{\partial u}{\partial\zeta} \le 0 \quad \text{ in } \mathbb{R}^N_+.
		\]
		
		Therefore, $u$ is constant in direction $\zeta$. Since $\zeta$ is arbitrary, we deduce that $u$ does not depend on $x'$. Hence $u$ depends only on $x_N$ and is monotone increasing in $x_N$.
	\end{proof}
	
	\subsection{Sign-changing nonlinearity}
	In this subsection, we consider the case that $f$ is sign-changing and $\frac{2N+2}{N+2}<p<2$.
	
	In this case, the strong comparison principle does not hold in all of $\mathbb{R}^N_+$. Hence, a delicate analysis of the critical set of solutions plays a vital role.
	To this end, we denote
	\begin{align*}
		Z_u&=\{x\in\mathbb{R}^N_+ \mid \nabla u(x) = 0\},\\
		Z_{u_\lambda^\nu}&=\{x\in\mathbb{R}^N_+ \mid \nabla u_\lambda^\nu(x) = 0\},\\
		Z_{f(u)}&=\{x\in\mathbb{R}^N_+ \mid u(x) \in Z_f\},\\
		Z_{f(u_\lambda^\nu)}&=\{x\in\mathbb{R}^N_+ \mid u_\lambda^\nu(x) \in Z_f\}.
	\end{align*}
	Motivated by \cite[Proposition 4.3]{MR4377321} and \cite[Lemma 13]{MR4771444}, we prove the following strong comparison type principle.
	\begin{lemma}\label{lem:limit_case}
		Assume $\frac{2N+2}{N+2}<p<2$ and $f:(0,+\infty)\to\mathbb{R}$ is a locally Lipschitz continuous function with $Z_f$ being a discrete set. Let $\Sigma=\mathbb{R}^N_+\setminus\Sigma_{\tilde{\lambda}}$ for some $\tilde{\lambda}>0$ and $u \in C^1_{\rm loc}(\Sigma) \cap C(\overline{\Sigma})$ be a solution to the problem
		\[
		\begin{cases}
			-\Delta_p u = f(u) & \text{ in } \Sigma,\\
			u>0 & \text{ in } \Sigma.
		\end{cases}
		\]
		Furthermore, assume that
		\begin{equation}\label{u_monotone}
			u \le u_\lambda^\nu \text{ in }\Sigma \quad\text{ for all } \lambda \ge \lambda_0,
		\end{equation}
		\begin{equation}\label{u_monotone2}
			u < u_{\lambda_0}^\nu \text{ on } \partial\Sigma,
		\end{equation}
		where $\nu\in S^{N-1}_+$ and $\lambda_0>0$.
		Then
		\[
		u < u_{\lambda_0}^\nu \text{ in } \Sigma \setminus \left(Z_{f(u)} \cap Z_{f(u_{\lambda_0}^\nu)} \cap Z_u \cap Z_{u_{\lambda_0}^\nu}\right).
		\]
	\end{lemma}
	
	\begin{proof}
		The proof follows the technique in \cite[Lemma 13]{MR4771444} (see also \cite{MR4635360,MR4377321}).
		We denote all zeroes of $f$ by
		\[
		0<z_1<z_2<\cdots
		\]
		We also denote $z_0=0$.		
		By contradiction, assume that there exists
		\[
		x_0 \in \Sigma \setminus \left(Z_{f(u)} \cap Z_{f(u_{\lambda_0}^\nu)} \cap Z_u \cap Z_{u_{\lambda_0}^\nu}\right)
		\]
		such that $u(x_0) = u_{\lambda_0}^\nu(x_0)$. There are four cases to be considered:
		
		$\circ$ \textit{Case 1:} $x_0 \in \Sigma \setminus Z_{f(u)}$. That is,
		\[
		z_k < u(x_0) < z_{k+1} \text{ for some } k\ge0.
		\]
		
		Let $\Omega_0$ be the connected component of $\Sigma\setminus Z_{f(u)}$ containing $x_0$. Then for all $x\in\partial\Omega_0$, we have either $u(x)=z_k$ or $u(x)=z_{k+1}$. By Theorem \ref{th:scp}, since $u(x_0) = u_{\lambda_0}^\nu(x_0)$, we have
		\begin{equation}\label{u_equal}
			u = u_{\lambda_0}^\nu \quad\text{ in } \Omega_0.
		\end{equation}
		Because $\Omega_0$ is open, there exists $r_0>0$ such that
		\[
		B_{2r_0}(x_0) \subset \Omega_0.
		\]
		We slide the ball $B_{r_0}(x_0)$ in $\Omega_0$, towards to $\partial\Sigma$ in direction $-\nu$ and keep its center on the ray $\Gamma_{x_0}:=\{x_0-t\nu \mid t\ge0\}$. The ball will touch $\partial\Omega_0$ for the first time at some point $\hat x_0\in\partial\Omega_0$. We denote by $\tilde x_0=x_0-t_0\nu$ the new center of the slid ball.
		
		Using \eqref{u_monotone} and \eqref{u_equal}, for every $x\in B_{r_0}(\tilde x_0)$, which is the slid point of $x+t_0\nu\in B_{r_0}(x_0)$, we have
		\[
		z_k < u(x) \le u_{\lambda_0+t_0}^\nu(x) =  u_{\lambda_0}^\nu(x+t_0\nu) = u(x+t_0\nu) \le \max_{\overline{B_{r_0}(x_0)}}u < z_{k+1}.
		\]
		Therefore, the touching point $\hat x_0$ must satisfy $u(\hat x_0)=z_k$. Moreover, by continuity, we have $u(\hat x_0) = u_{\lambda_0}^\nu(\hat x_0)$. We consider two possibilities.
		
		- Possibility (i): $\hat x_0\in\partial\Sigma$. Then $u(\hat x_0) = u_{\lambda_0}^\nu(\hat x_0) = u(\hat x_0 + \lambda_0\nu)$ contradicts assumption \eqref{u_monotone2}.
		
		- Possibility (ii): $\hat x_0\notin\partial\Sigma$. Let us define the function
		\[
		w(x) := u(x) - z_k \quad\text{ for }x\in B_{r_0}(\tilde x_0).
		\]
		
		Since $p<2$ and $f$ is locally Lipschitz continuous in $(0,+\infty)$, we have
		\[
		C w^{p-1} + f(u) = C w^{p-1} + f(u) - f(z_k) \ge C w^{p-1} - K(u-z_k) \ge 0
		\]
		for sufficiently large $C$.
		Hence $w$ satisfies
		\[
		\begin{cases}
			-\Delta_p w + C w^{p-1} \ge 0 &\text{ in } B_{r_0}(\tilde x_0),\\
			w>0 &\text{ in } B_{r_0}(\tilde x_0),\\
			w(\hat x_0)=0.
		\end{cases}
		\]
		
		By Hopf's lemma \cite{MR768629}, it follows that
		\begin{equation}\label{Hopf_deri}
			\frac{\partial u}{\partial\eta}(\hat x_0) < 0,
		\end{equation}
		where $\eta=\frac{\hat x_0 - \tilde x_0}{|\hat x_0 - \tilde x_0|}$ is the outward normal at $\hat x_0$. In particular, $|\nabla u(\hat x_0)|\ne0$. Since $u\in C^1(\mathbb{R}^N)$, there exists a ball $B_{\rho_0}(\hat x_0) \subset \mathbb{R}^N$ such that $|\nabla u|\ne0$ in $B_{\rho_0}(\hat x_0)$. By Theorem \ref{th:cscp}, since $u(\hat x_0) = u_{\lambda_0}^\nu(\hat x_0)$, we have
		\[
		u = u_{\lambda_0}^\nu \quad\text{ in } B_{\rho_0}(\hat x_0).
		\]
		From \eqref{Hopf_deri}, we can find a point $x_1 \in \{\hat x_0 + t \eta \mid t>0\} \cap B_{\rho_0}(\hat x_0)$ which is close to $\hat x_0$ such that
		\[
		z_{k-1} < u(x_1) < u(\hat x_0) = z_k.
		\]
		
		Therefore, from a point $x_0$ with $u(x_0) = u_{\lambda_0}^\nu(x_0)$ and $z_k < u(x_0) < z_{k+1}$, we have found a new point $x_1$ satisfying $u(x_1) = u_{\lambda_0}^\nu(x_1)$ and $z_{k-1} < u(x_1) < z_k$. Repeating this argument a finite number of times, we finally find a ball that touches $\partial\Sigma$. Then we have a contradiction as in Possibility (i).
		
		$\circ$ \textit{Case 2:} $x_0 \in \Sigma \setminus Z_{f(u_{\lambda_0}^\nu)}$. Since $u(x_0) = u_{\lambda_0}^\nu(x_0)$, this case is actually Case 1.
		
		$\circ$ \textit{Case 3:} $x_0 \in \Sigma \setminus Z_u$.
		
		Since $u \in C^1(\overline{\Sigma})$, we deduce $|\nabla u| \ne 0$ in $B_\varepsilon(x_0)$ for some $\varepsilon>0$. Theorem \ref{th:cscp} now comes into play to yield
		\[
		u = u_{\lambda_0}^\nu \quad\text{ in } B_\varepsilon(x_0).
		\]
		
		Moreover, since $u$ is not constant in $B_\varepsilon(x_0)$, we can find $x_1\in B_\varepsilon(x_0)$ such that $x_1\notin Z_{f(u)}$. Using $u(x_1) = u_{\lambda_0}^\nu(x_1)$, we will reach a contradiction as in Case 1.
		
		$\circ$ \textit{Case 4:} $x_0 \in \Sigma \setminus Z_{u_{\lambda_0}^\nu}$. This case is similar to Case 3.
	\end{proof}
	
	We also recall the following weak comparison principle in a strip for solutions with small gradients from \cite[Proposition 11]{MR4771444} (see also a stronger version in \cite[Theorem 1.6]{MR3303939}).
	\begin{proposition}[Proposition 11 in \cite{MR4771444}]\label{prop:degenerate}
		Let $1<p<2$ and $f:(0,+\infty)\to\mathbb{R}$ be a locally Lipschitz continuous function.
		Let $M,a,b>0$ and let $u\in C^1(\overline{\Omega})$ be a subsolution and $v\in C^1(\overline{\Omega})$ be a supersolution to
		\[
		\begin{cases}
			-\Delta_p w = f(w) &\text{ in } \Omega,\\
			a<w<b &\text{ in } \Omega,
		\end{cases}
		\]
		where
		\[
		\Omega \subset \{x\in\mathbb{R}^N \mid 0 \le x_N \le M\}.
		\]
		Assume that
		\[
		u \le v \quad\text{ on } \partial \Omega
		\]
		and
		\[
		|\nabla u| + |\nabla v| < \eta \quad\text{ in } \Omega,
		\]
		where $\eta>0$.
		Then there exists $\eta_0=\eta_0(N,p,f,M,a,b)>0$ such that $u \le v$ in $\Omega$ whenever $\eta\le\eta_0$.
	\end{proposition}
	
	We are in a position to prove Theorem \ref{th:rigidity2}.
	\begin{proof}[Proof of Theorem \ref{th:rigidity2}]
		We may proceed as in the proof of Theorem \ref{th:rigidity1} to deduce that the set
		\[
		\Lambda = \{ \lambda>0 \mid u \le u_\mu^\nu \text{ in } \mathbb{R}^N_+ \text{ for all } \mu > \lambda \}
		\]
		is nonempty. We will show that
		\[
		\lambda_0 := \inf \Lambda = 0.
		\]
		
		Assume, on contrary, that $\lambda_0>0$. By continuity of $u$, we have $u \le u_{\lambda_0}^\nu$ in $\mathbb{R}^N_+$. To reach a contradiction, we will search for some $0<\varepsilon_0<\frac{\lambda_0}{2}$ such that
		\begin{equation}\label{key2}
			u \le u_\lambda^\nu \quad\text{ in } \mathbb{R}^N_+
		\end{equation}
		for all $\lambda \in (\lambda_0-\varepsilon_0, \lambda_0)$.
		
		$\circ$ Due to Lemma \ref{lem:positive}, there exist $\tilde{\lambda},\tilde{\delta}>0$ sufficiently small such that
		\begin{equation}\label{u0}
			u+\tilde{\delta}\leq u_\lambda^\nu \quad\text{ in } \Sigma_{\tilde{\lambda}}
		\end{equation}
		for all $\lambda>\lambda_0/2$.
		We decompose $\mathbb{R}^N_+$ into four disjoint subsets
		\[
		\mathbb{R}^N_+ = \Sigma_{\tilde{\lambda}} \cup (\Sigma\setminus \Omega_\lambda) \cup (\mathbb{R}^N_+\setminus\Sigma_M )\cup  \Omega_\lambda,
		\]
		where
		\begin{align*}
			\Sigma &= \Sigma_M\setminus\Sigma_{\tilde{\lambda}},\\
			\Omega_\lambda &= \{x\in\mathbb{R}^N_+ \mid |\nabla u| + |\nabla u_\lambda^\nu| + \min_{z\in Z_f}|u-z| + \min_{z\in Z_f}|u_\lambda^\nu-z| < \eta\}.
		\end{align*}
		Here $\eta>0$ is sufficiently small and $M>\lambda^*$ is sufficiently large such that $\Omega_\lambda\subset\Sigma$ for all $\lambda \in (\frac{\lambda_0}{2}, \lambda_0)$. (We recall that $\lambda^*$ is defined as in \eqref{b}.) Notice that such $\eta$ and $M$ can be chosen by combining the fact $\min_{z\in Z_f}|u-z|<\eta$ in $\Omega_\lambda$ with \eqref{uni_large} and $Z_f\cap(t_0,+\infty)=\emptyset$. Then we choose $\eta$ even smaller if necessary, such that Proposition \ref{prop:degenerate} holds.
		
		$\circ$ We claim that
		\begin{equation}\label{u1}
			u \le u_\lambda^\nu \quad\text{ in } \Sigma\setminus \Omega_\lambda
		\end{equation}
		for all $\lambda \in (\lambda_0-\varepsilon_0, \lambda_0)$, where $\varepsilon_0<\frac{\lambda_0}{2}$ is sufficiently small.
		
		Assume that \eqref{u1} does not hold. Then there exist two sequences $\lambda_n \nearrow \lambda_0$ and $x_n := (x'_n, (x_n)_N) \in \mathbb{R}^{N-1} \times [\tilde{\lambda},M)$ such that
		\begin{equation}\label{u2}
			u(x_n) > u_{\lambda_n}^\nu(x_n)
		\end{equation}
		and
		\begin{equation}\label{u4}
			|\nabla u(x_n)| + |\nabla u_{\lambda_n}^\nu(x_n)| + \min_{z\in Z_f}|u(x_n)-z| + \min_{z\in Z_f}|u_{\lambda_n}^\nu(x_n)-z| \ge \eta.
		\end{equation}
		Moreover, we may assume $(x_n)_N \to y_0 \in [\tilde{\lambda},M]$. Now we set
		\[
		v_n(x', x_N) = u(x'+x'_n, x_N).
		\]
		
		As in the proof of Theorem \ref{th:rigidity1}, we have
		\[
		v_n \to v \quad\text{ in } C^1_{\rm loc}(\mathbb{R}^N_+\setminus\Sigma_{\tilde{\lambda}})
		\]
		up to a subsequence. Moreover, $v$ weakly solves $-\Delta_p v = f(v)$ in $\mathbb{R}^N_+\setminus\Sigma_{\tilde{\lambda}}$.
		Using the definition of $\lambda_0$ and passing \eqref{u2} and \eqref{u4} to the limit, we also obtain
		\begin{align*}
			&v \le v_\lambda^\nu \quad\text{ in } \mathbb{R}^N_+ \quad\text{ for all } \lambda \ge \lambda_0,\\
			&v(x_0) = v_{\lambda_0}^\nu(x_0),\\
			&|\nabla v(x_0)| + |\nabla v_{\lambda_0}^\nu(x_0)| + \min_{z\in Z_f}|v(x_0)-z| + \min_{z\in Z_f}|v_{\lambda_0}^\nu(x_0)-z| \ge \eta,
		\end{align*}
		where $x_0 := (0,y_0) \in\overline{\Sigma_M\setminus\Sigma_{\tilde{\lambda}}}$. Moreover, \eqref{u0} implies $x_0 \in \overline{\Sigma_M}\setminus\Sigma_{\tilde{\lambda}}$. However, the existence of such a solution $v$ and point $x_0$ contradicts Lemma \ref{lem:limit_case}.		
		Therefore, \eqref{u1} must hold.
		
		$\circ$ Next, we show that
		\begin{equation}\label{u5}
			u \le u_\lambda^\nu \quad\text{ in } \mathbb{R}^N_+\setminus\Sigma_M
		\end{equation}
		for all $\lambda \in (\lambda_0-\varepsilon_0, \lambda_0)$.
		From \eqref{u1} and $\Omega_\lambda\subset\Sigma$, we already have $u \le u_\lambda^\nu$ on $\partial (\mathbb{R}^N_+\setminus\Sigma_M )$. Moreover, from \eqref{b}, we have $u_\lambda^\nu(x) > t_0$ for each $x\in\mathbb{R}^N_+\setminus\Sigma_M$. Hence \eqref{u5} follows by applying  Proposition \ref{prop:wcp2} with $u$ and $v:=u_\lambda^\nu$ on $\mathbb{R}^N_+\setminus\Sigma_M$.
		
		$\circ$ From \eqref{u1}, we also have $u \le u_\lambda^\nu$ on $\partial \Omega_\lambda$. Therefore, we can apply Proposition \ref{prop:degenerate} for $v=u_\lambda^\nu$ to deduce
		\begin{equation}\label{u3}
			u \le u_\lambda^\nu \quad\text{ in } \Omega_\lambda.
		\end{equation}
		
		Combining \eqref{u0}, \eqref{u1}, \eqref{u5}, \eqref{u3}, we obtain \eqref{key2}. This contradicts the definition of $\lambda_0$. Hence $\lambda_0=0$.
		
		Finally, arguing as in the proof of Theorem \ref{th:rigidity1}, we conclude that $u$ depends only on $x_N$ and is monotone increasing in $x_N$.
	\end{proof}
	
	\begin{proof}[Proof of Proposition \ref{prop:rigidity}]
		If $u$ is a bounded solution, then by \cite[Proposition 2.2]{MR2056284}, we have $u<u\le t_0$. The strong comparison principle implies $u<u<t_0$. Now Lemma \ref{lem:positive} gives $\lim_{x_N\to0} u(x)= t_0$ uniformly in $x'\in\mathbb{R}^{N-1}$. Then, using (iii) and arguing as in Theorem \ref{th:rigidity1}, we get the 1D symmetry and monotonicity of $u$.
		
		Now if \eqref{bounded_cond} is satisfied, then by exploiting \cite[Proposition 2.3]{MR2056284} we deduce that $u$ is bounded and the conclusion follows as before.
	\end{proof}
	
	\section{On the pure singular problem}\label{sect5}
	In this section, we deal with problem \eqref{pure}. First of all, we prove Theorem \ref{th:decreasing} since we need the monotonicity result in analyzing problem \eqref{pure} in dimension one.
	
	\subsection{Strictly decreasing nonlinearity}
	\begin{proof}[Proof of Theorem \ref{th:decreasing}]
		Since $\lim_{t\to0^+} f(t) > 0$, we have $f(t) > c_0$ in $(0, t_0)$ for some $t_0,c_0>0$.		
		Let any $\lambda>0$ and set $t_1=\|u\|_{L^\infty(\Sigma_\lambda)} + 1$.
		Using Lemma \ref{lem:lower_bound}, one may check that
		\[
		\begin{cases}
			-\Delta_p u = f(u) &\text{ in } \Sigma_\lambda,\\
			-\Delta_p u_\lambda = f(u_\lambda) &\text{ in } \Sigma_\lambda,\\
			0<u<t_1 &\text{ in } \Sigma_\lambda,\\
			\min\{C \lambda, t_0\} \le u_\lambda\le\|u\|_{L^\infty(\Sigma_{2\lambda})} &\text{ in } \Sigma_\lambda,\\
			u \le u_\lambda &\text{ on } \partial\Sigma_\lambda,
		\end{cases}
		\]
		where $t_0>0$.
		Hence, Proposition \ref{prop:wcp} yields $u\leq u_\lambda$ in $\Sigma_\lambda$ for all $\lambda>0$. Therefore, $u$ is monotone increasing in $x_N$.
		
		Now suppose that \eqref{uni_gradient_zero} holds. Let any $\nu\in \mathbb{S}^{N-1}_+$ and $\lambda>0$. From \eqref{uni_gradient_zero} and the mean value theorem, we deduce
		\[
		\lim_{x_N\to+\infty} (u-u_\lambda^\nu) = 0 \text{ uniformly in } x'\in\mathbb{R}^{N-1}.
		\]
		Set $t_2 = \frac{1}{2}\min\{C \lambda, t_0\}$, we have
		\[
		\begin{cases}
			-\Delta_p u = f(u) &\text{ in } \mathbb{R}^N_+,\\
			-\Delta_p u_\lambda^\nu = f(u_\lambda^\nu) &\text{ in } \mathbb{R}^N_+,\\
			u>0 &\text{ in } \mathbb{R}^N_+,\\
			u_\lambda^\nu>t_2 &\text{ in } \mathbb{R}^N_+,\\
			u \le u_\lambda^\nu &\text{ on } \partial\mathbb{R}^N_+.
		\end{cases}
		\]
		Hence we can apply Proposition \ref{prop:wcp2} to derive
		\[
		u \le u_\lambda^\nu \,\text{ in } \mathbb{R}^N_+ \quad\text{ for all } \lambda>0.
		\]
		As in the proof of Theorem \ref{th:rigidity2}, this implies that $u$ depends only on $x_N$ and is monotone increasing in $x_N$.
		
		If, instead of \eqref{uni_gradient_zero}, we assume $\limsup_{t\to+\infty} f'(t) < 0$ and $|\nabla u| \in L^\infty(\mathbb{R}^N\setminus\Sigma_\lambda)$ for all $\lambda>0$. Then we can exploit Proposition \ref{prop:wcp3} to get the thesis.
	\end{proof}
	
	\subsection{Dimension one}
	
	In this subsection, we classify all solutions to the ODE problem
	\begin{equation}\label{ode}
		\begin{cases}
			-(|v'|^{p-2} v')' = \dfrac{1}{v^\gamma} & \text{ in } \mathbb{R}_+,\\
			v(t)>0 & \text{ in } \mathbb{R}_+,\\
			v(0)=0,
		\end{cases}
	\end{equation}
	where $p>1$ and $\gamma>0$.
	
	\begin{proof}[Proof of Theorem \ref{th:1D}]
		By Theorem \ref{th:decreasing}, we know that $v'\ge0$ in $\mathbb{R}_+$. We show that actually $v'>0$ in $\mathbb{R}_+$.
		
		Clearly, there exists at least one $t_0\in\mathbb{R}_+$ such that $v'(t_0)>0$. Let $(a,b)$ be the maximal interval containing $t_0$ such that $v'>0$ in $(a,b)$. We need to show that $a=0$ and $b=+\infty$. We only prove the latter since the former can be done similarly.
		
		Assume now that $b<+\infty$ and $v'(b)=0$. The case $v'=0$ in $(b,+\infty)$ cannot happen since it contradicts the first equation of \eqref{ode}. Hence, there exists some $t_1>b$ with $v'(t_1)>0$. Consider the maximal interval $(t_2,t_1] \subset (b, t_1]$ such that
		\[
		v' > 0 \text{ in } (t_2,t_1], \quad v'(t_2)=0.
		\]
		From the standard elliptic regularity, we know that $v$ is $C^2$ in $(t_2,t_1]$. Hence, in this interval, it holds
		\begin{equation}\label{kt}
			(p-1) (v')^{p-2} v'' = -\frac{1}{v^\gamma}.
		\end{equation}
		Hence $v''(t) < 0$ in $(t_2,t_1]$. This implies $v'(t_2) > v'(t_1)$, which is a contradiction.
		
		We have showed that $v'>0$ in $\mathbb{R}_+$. Hence $v\in C^2(\mathbb{R}_+)$ by the standard elliptic regularity and $v$ verifies \eqref{kt} in $\mathbb{R}_+$. This implies
		\begin{align}
			\frac{p-1}{p} (v')^p - \frac{v^{1-\gamma}}{\gamma-1} = M \text{ in } \mathbb{R}_+ &\quad\text{ if } 0<\gamma\ne1,\label{k3}\\
			\frac{p-1}{p} (v')^p + \ln v = M \text{ in } \mathbb{R}_+ &\quad\text{ if } \gamma=1,\label{k4}
		\end{align}
		for some constant $M$. On the other hand, by Lemma \ref{lem:lower_bound}, we have $\lim_{t\to+\infty}v(t)=+\infty$. Therefore, both \eqref{k3} with $0<\gamma<1$ and \eqref{k4} yield a contradiction by letting $t\to+\infty$. This means that \eqref{ode} has no solution for $0<\gamma<1$.
		
		In what follows, we assume $\gamma>1$. Letting $t\to+\infty$ in \eqref{k3}, we deduce $M\ge0$. Now we rewrite \eqref{k3} as
		\[
		\left(M + \frac{v^{1-\gamma}}{\gamma-1}\right)^{-\frac{1}{p}} v' = \left(\frac{p}{p-1}\right)^\frac{1}{p} \quad\text{ in } \mathbb{R}_+.
		\]
		By integrating and using $v(0)=0$, this gives
		\begin{equation}\label{gh}
			\int_{0}^{v(t)} \left(M + \frac{s^{1-\gamma}}{\gamma-1}\right)^{-\frac{1}{p}} ds = \left(\frac{p}{p-1}\right)^\frac{1}{p}t \quad\text{ for all } t\in\mathbb{R}_+.
		\end{equation}
		
		Conversely, for every $M\ge0$ we have
		\[
		\int_{0}^{+\infty} \left(M + \frac{s^{1-\gamma}}{\gamma-1}\right)^{-\frac{1}{p}} ds = +\infty \text{ and }
		\int_{0}^{t} \left(M + \frac{s^{1-\gamma}}{\gamma-1}\right)^{-\frac{1}{p}} ds < +\infty \text{ for all } t>0.
		\]
		Therefore, for each $M\ge0$, formula \eqref{gh} uniquely determines a function $v:=v_M$ which is a solution to \eqref{ode}. Using \eqref{k3}, we see that these solutions are characterized by the limit
		\[
		\lim_{t\to+\infty} v_M'(t) = \left(\frac{Mp}{p-1}\right)^\frac{1}{p}.
		\]
		
		When $M=0$, a direct calculation yields
		\[
		v_0(t) = \left[\frac{(\gamma+p-1)^p}{p^{p-1}(p-1)(\gamma-1)}\right]^\frac{1}{\gamma+p-1}	t^\frac{p}{\gamma+p-1}.
		\]		
		Moreover, by change of variables in \eqref{gh}, we can show that all other solutions are related to each other via the formula $v_M(t) \equiv \lambda^{-\frac{p}{\gamma+p-1}} v_1(\lambda t)$, where $M=\lambda^\frac{(\gamma-1)p}{\gamma+p-1}$.
		
		This completes the proof.
	\end{proof}
	
	\subsection{Higher dimensions}
	We employ the technique from \cite[Proposition 5]{MR4753083} to show that solution $u$ of problem \eqref{pure} grows at most at a linear rate as $x_N \to +\infty$.
	\begin{lemma}\label{lem:upper_bound_infty}
		Let $1<p\le N$, $\gamma>0$ and let $u \in C^{1,\alpha}_{\rm loc}(\mathbb{R}^N_+) \cap C(\overline{\mathbb{R}^N_+})$ be a solution to problem \eqref{pure} with $u\in L^\infty(\Sigma_{\overline\lambda})$ for some $\overline\lambda>0$.
		Then there exists a positive constant $C = C(p, \gamma, \theta, N)$ such that
		\[
		u(x) \le C x_N \quad\text{ in } \mathbb{R}^N_+\setminus\Sigma_{\overline{\lambda}}.
		\]
	\end{lemma}
	\begin{proof}	
		If $u$ is a solution to \eqref{pure}, then
		\begin{equation}\label{v}
			v(x) := \left(\frac{\overline\lambda}{2}\right)^{-\frac{p}{\gamma+p-1}} u\left(\frac{\overline\lambda}{2}x\right)
		\end{equation}
		is also a solution. Therefore, we may assume that our solution $u$ is bounded in the strip $\Sigma_2$.
		
		Let any $x_0 = (x_0', x_{0,N}) \in \mathbb{R}^N_+$ with $x_{0,N}:=4R>2$. We set
		\[
		u_R(x) := R^{-\frac{p}{\gamma+p-1}} u(x_0+R(x-x_0)),
		\]
		then $u_R>0$ in $B_4(x_0)$ and
		\[
		-\Delta_p u_R = \frac{1}{u_R^{\gamma+1}} u_R \quad\text{ in } B_4(x_0).
		\]
		
		By Lemma \ref{lem:lower_bound}, we have
		\[
		u_R^{\gamma+1}(x) = \left(R^{-\frac{p}{\gamma+p-1}} u(x_0+R(x-x_0))\right)^{\gamma+1} \ge C^{\gamma+1} \quad\text{ in } B_2(x_0).
		\]
		Hence
		\[
		c(x) := \frac{1}{u_R^{\gamma+1}} \le \frac{1}{C^{\gamma+1}},
		\]
		where $C$ is independent of $x_0$.		
		By Harnack's inequality, we have
		\[
		\sup_{B_1(x_0)} u_R \le C_H \inf_{B_1(x_0)} u_R,
		\]
		where $C_H=C_H(N,p,\gamma)$. In particular, by setting $u_0=u(x_0)$, we have
		\[
		u_0 \le \sup_{B_R(x_0)} u = R^{\frac{p}{\gamma+p-1}} \sup_{B_1(x_0)} u_R \le C_H R^{\frac{p}{\gamma+p-1}} \inf_{B_1(x_0)} u_R = C_H \inf_{B_R(x_0)} u \le C_H u(x)
		\]
		for all $x\in B_R(x_0)$. Hence
		\[
		u(x) \ge C_H^{-1} u_0 \quad\text{ on } \partial B_R(x_0).
		\]
		
		Now we consider, for $p<N$, the fundamental solution of the $p$-Laplace operator
		\[
		v_{c,k} = c\left(\frac{1}{|x-x_0|^\frac{N-p}{p-1}}+k\right),
		\]
		which satisfies
		\[
		\Delta_p v_{c,k} = 0 \quad\text{ in } \mathbb{R}^N_+\setminus\{x_0\}
		\]
		for all $c,k\in\mathbb{R}$. We can choose $c,k$ such that
		\[
		\begin{cases}
			v_{c,k} = C_H^{-1} u_0 & \text{ on }  \partial B_R(x_0),\\
			v_{c,k} = 0 & \text{ on }  \partial B_{4R}(x_0).
		\end{cases}
		\]
		More precisely, the above condition is fulfilled with
		\[
		c=\frac{C_H^{-1}u_0(4R)^\frac{N-p}{p-1}}{4^\frac{N-p}{p-1}-1} := \tilde{c}u_0R^\frac{N-p}{p-1} \quad\text{ and }\quad k=-\frac{1}{(4R)^\frac{N-p}{p-1}}.
		\]
		
		Using $(v_{c,k}-u-\varepsilon)^+$, where $\varepsilon>0$, as a test function in $\Delta_p v_{c,k} = 0$ and $-\Delta_p u = \frac{1}{u^\gamma}$, we get
		\begin{align*}
			&\int_{B_{4R}(x_0) \setminus B_R(x_0)} (|\nabla v_{c,k}|^{p-2}\nabla v_{c,k} - |\nabla u|^{p-2}\nabla u, \nabla (v_{c,k}-u-\varepsilon)^+) \\
			&= \int_{B_R(x_0)} \left(0 - \frac{1}{u^\gamma}\right) (v_{c,k}-u)^+ \le 0.
		\end{align*}
		Hence $v_{c,k} \le u+\varepsilon$ in $B_{4R}(x_0) \setminus B_R(x_0)$ for all $\varepsilon>0$. Therefore,
		\[
		u \ge v_{c,k} \quad \text{ in } B_{4R}(x_0) \setminus B_R(x_0).
		\]
		In particular,
		\begin{align*}
			u(x_0',1) &\ge v_{c,k}(x_0',1)\\
			&= c\left(\frac{1}{|(x_0',1)-(x_0',x_{0,N})|^\frac{N-p}{p-1}}+k\right)\\
			&= \tilde{c}u_0R^\frac{N-p}{p-1} \left(\frac{1}{|4R-1|^\frac{N-p}{p-1}}-\frac{1}{(4R)^\frac{N-p}{p-1}}\right).
		\end{align*}
		By the mean value theorem for the function $h(t)=\frac{1}{t^\frac{N-p}{p-1}}$, we have
		\[
		\frac{1}{|4R-1|^\frac{N-p}{p-1}}-\frac{1}{(4R)^\frac{N-p}{p-1}} \ge \frac{N-p}{p-1}\frac{1}{(4R)^\frac{N-1}{p-1}}.
		\]
		Therefore,
		\[
		u(x_0',1) \ge \frac{N-p}{p-1} \frac{ \tilde{c}u_0}{4^\frac{N-1}{p-1}R}.
		\]
		
		Since $u \in L^\infty(\Sigma_2)$, we deduce
		\[
		u(x_0) = u_0 \le CR,
		\]
		where $C$ does not depend on $R$. Since $x_0$ is arbitrary and $x_{0,N}=4R$, we obtain that
		\[
		u(x) \le \frac{C}{4} x_N \quad\text{ in } \mathbb{R}^N_+\setminus\Sigma_2.
		\]
		
		Scaling back, using \eqref{v}, we obtain the thesis for $p < N$.
		
		The	case $p = N$ follows by repeating the same argument but replacing the
		fundamental solutions with the logarithmic one	
		\[
		w_{c,k} = c\left(k-\ln|x-x_0|\right).
		\]
		More precisely, by choosing
		\[
		c=\frac{C_H^{-1}u_0}{\ln 4} := \hat{c}u_0 \quad\text{ and }\quad k=\ln(4R),
		\]
		we have
		\[
		\begin{cases}
			w_{c,k} = C_H^{-1} u_0 & \text{ on }  \partial B_R(x_0),\\
			w_{c,k} = 0 & \text{ on }  \partial B_{4R}(x_0).
		\end{cases}
		\]
		Hence, as before
		\[
		u \ge w_{c,k} \quad \text{ in } B_{4R}(x_0) \setminus B_R(x_0).
		\]
		In particular,
		\begin{align*}
			u(x_0',1) &\ge w_{c,k}(x_0',1)\\
			&= c\left(k-\ln|(x_0',1)-(x_0',x_{0,N})|\right)\\
			&= \hat{c}u_0\left(\ln(4R)-\ln(4R-1)\right)\\
			&\ge \frac{\hat{c}u_0}{4R}.
		\end{align*}
		Then we get the thesis as in the previous case.
		This completes the proof.
	\end{proof}
	
	Given the previous asymptotic bounds on \( u \), we can apply the scaling technique as in \cite[Proposition 7]{MR4753083} to establish a bound on the gradient.
	\begin{lemma}
		Let $1<p\le N$, $\gamma\ge1$ and let $u \in C^{1,\alpha}_{\rm loc}(\mathbb{R}^N_+) \cap C(\overline{\mathbb{R}^N_+})$ be a solution to problem \eqref{pure} with $u\in L^\infty(\Sigma_{\overline\lambda})$ for some $\overline\lambda>0$. Then there exists a positive constant $C_\lambda > 0$ such that
		\[
		|\nabla u(x)| \le C_\lambda \quad\text{ in } \mathbb{R}^N_+\setminus\Sigma_\lambda
		\]
		for all $\lambda>0$.
	\end{lemma}
	\begin{proof}
		Let $x_0\in\mathbb{R}^N_+\setminus\Sigma_\lambda$ and set $R=x_{0,N}>\lambda$. We define
		\[
		u_R(x) := \frac{u(Rx)}{R} \quad\text{ in } B_{\frac{1}{2}}\left(\frac{x_0}{R}\right).
		\]
		By Lemma \ref{lem:upper_bound_infty} we have $u_R \le C_\lambda$. Moreover, from Lemma \ref{lem:lower_bound}, we deduce
		\[
		-\Delta_p u_R = \frac{R}{u^\gamma(Rx)} \le CR^{-\frac{(\gamma-1)(p-1)}{\gamma+p-1}} < C_\lambda' \quad\text{ in } B_{\frac{1}{2}}\left(\frac{x_0}{R}\right),
		\]
		where $C_\lambda'$ is independent of $x_0$.		
		By the standard gradient estimate, we have $|\nabla u_R| \le C_\lambda''$ in $B_{\frac{1}{4}}\left(\frac{x_0}{R}\right)$. This indicates $|\nabla u| \le C_\lambda''$ in $B_{\frac{R}{4}}(x_0)$. The thesis follows from the arbitrariness of $x_0$.
	\end{proof}
	
	Similarly, we have the following estimate on the gradient of solutions if \eqref{sublinear} holds.
	\begin{lemma}\label{lem:upper_bound_gradient2}
		Let $p>1$, $\gamma\ge1$  and let $u \in C^{1,\alpha}_{\rm loc}(\mathbb{R}^N_+) \cap C(\overline{\mathbb{R}^N_+})$ be a solution to problem \eqref{pure} satisfying \eqref{sublinear}. Then
		\[
		\lim_{x_N\to+\infty} |\nabla u(x',x_N)|=0 \text{ uniformly in } x'\in\mathbb{R}^{N-1}_+.
		\]
	\end{lemma}
	\begin{proof}
		Let $x_0\in\mathbb{R}^N_+\setminus\Sigma_\lambda$ and set $R=x_{0,N}>\lambda$. Then let $\varepsilon>0$ and define
		\[
		u_R(x) := \frac{u(Rx)}{\varepsilon R} \quad\text{ in } B_{\frac{1}{2}}\left(\frac{x_0}{R}\right).
		\]
		
		By assumption \eqref{sublinear}, there exists $R_\varepsilon>0$ such that $u_R \le 1$ in $B_{\frac{1}{2}}\left(\frac{x_0}{R}\right)$ whenever $R>R_\varepsilon$. Moreover, from Lemma \ref{lem:lower_bound}, we deduce
		\[
		-\Delta_p u_R = \frac{R}{\varepsilon^{p-1} u^\gamma(Rx)} \le \frac{C}{\varepsilon^{p-1}}R^{-\frac{(\gamma-1)(p-1)}{\gamma+p-1}} < C' \quad\text{ in } B_{\frac{1}{2}}\left(\frac{x_0}{R}\right),
		\]
		where $C'=C'(\varepsilon,R_\varepsilon)$ is independent of $R>R_\varepsilon$.		
		By the standard gradient estimate, we have $|\nabla u_R| \le C$ in $B_{\frac{1}{4}}\left(\frac{x_0}{R}\right)$. This indicates $|\nabla u| \le C\varepsilon$ in $B_{\frac{R}{4}}(x_0)$ for $R>R_\varepsilon$ and the thesis follows.
	\end{proof}
	
	Now we can conclude the last main result of this paper, namely, Theorem \ref{th:pure}.
	\begin{proof}[Proof of Theorem \ref{th:pure}]
		For all $\gamma>0$,	Proposition \ref{prop:zero_convergence} and Lemma \ref{lem:upper_bound_infty} imply that \eqref{bound_strips} holds. Estimate \eqref{asymptotic} is provided by Lemma \ref{lem:upper_bound_infty}. Now we differentiate between two cases:
		
		\textit{Case 1:} $\gamma>1$. By Lemma \ref{lem:upper_bound_gradient2}, condition \eqref{uni_gradient_zero} is fulfilled if \eqref{sublinear} holds.		
		The conclusion now follows from Theorem \ref{th:decreasing} with the aid of Theorem \ref{th:1D}.
		
		\textit{Case 2:} $0<\gamma<1$. In this case, Lemma \ref{lem:lower_bound} yields $u(x) \ge x_N^\frac{p}{\gamma+p-1}$ in $\mathbb{R}^N_+$, which is contradict to \eqref{asymptotic}. Hence, such solutions do not exist.
	\end{proof}
	
	\subsection{The conformal case} In this subsection, we study problem \eqref{pure} with $p=N$, namely,
	\begin{equation}\label{conformal}	
		\begin{cases}
			-\Delta_N u = \dfrac{1}{u^\gamma} & \text{ in } \mathbb{R}^N_+,\\
			u>0 & \text{ in } \mathbb{R}^N_+,\\
			u=0 & \text{ on } \partial\mathbb{R}^N_+
		\end{cases}
	\end{equation}
	This case is usually referred to as the conformal case, since $\Delta_N$ is invariant under the Kelvin transform
	\[
	\hat u(x) := u\left(\frac{x}{|x|^2}\right).
	\]
	Formally, we have $\Delta_N \hat u = \frac{1}{|x|^{2N}} (\Delta_N u) \left(\frac{x}{|x|^2}\right)$.
	
	\begin{proof}[Proof of Theorem \ref{th:pure_N}]
		Since $u \in C^{1,\alpha}_{\rm loc}(\mathbb{R}^N_+) \cap C(\overline{\mathbb{R}^N_+})$ is a solution to \eqref{conformal}, then one can verifies that $\hat u \in C^{1,\alpha}_{\rm loc}(\mathbb{R}^N_+) \cap C(\overline{\mathbb{R}^N_+}\setminus\{0\})$ and $\hat u$ solves
		\begin{equation}\label{kelvin}
			\begin{cases}
				-\Delta_N \hat u = \dfrac{1}{|x|^{2N}\hat u^\gamma} & \text{ in } \mathbb{R}^N_+,\\
				\hat u>0 & \text{ in } \mathbb{R}^N_+,\\
				\hat u=0 & \text{ on } \partial\mathbb{R}^N_+\setminus\{0\}.
			\end{cases}
		\end{equation}
		Moreover,
		\[
		\lim_{|x|\to+\infty} \hat u(x) = \lim_{|x|\to0} u(x) = 0.
		\]
		Solutions to \eqref{kelvin} are still understood in the weak sense
		\[
		\int_{\mathbb{R}^N_+} (|\nabla \hat u|^{N-2} \nabla \hat u, \nabla \varphi) = \int_{\mathbb{R}^N_+} \frac{\varphi}{|x|^{2N}\hat u^\gamma} \quad\text{ for all } \varphi\in C^1_c(\mathbb{R}^N_+).
		\]
		
		For any $\lambda<0$, we denote $\hat\Sigma_\lambda=\{x:=(x_1,\tilde x)\in\mathbb{R}^N_+ \mid x_1<\lambda\}$, $x_\lambda=(2\lambda-x_1, \tilde x)$ and $\hat u_\lambda(x) = \hat u(x_\lambda)$ in $\mathbb{R}^N_+$. By the reflection invariance of the $N$-Laplacian, we deduce
		\begin{equation}\label{kelvin_r}
			-\Delta_N \hat u_\lambda = \frac{1}{|x_\lambda|^{2N}\hat u_\lambda^\gamma} \text{ in } \mathbb{R}^N_+
		\end{equation}
		in the weak sense. Let any $\varepsilon>0$. We can find a small $\delta>0$ such that $\hat u<\varepsilon$ in $\mathbb{R}^N_+\cap B_\delta(0_\lambda)$. Now we set $w=(\hat u-\hat u_\lambda-\varepsilon)^+$, then $w=0$ in $\mathbb{R}^N_+\cap B_\delta(0_\lambda)$. Moreover, since
		\[
		\lim_{|x|\to+\infty} (\hat u(x)-\hat u_\lambda(x)) = 0 \quad\text{ and }\quad \hat u(x)-\hat u_\lambda(x)=0 \text{ on } \partial\mathbb{R}^N_+\setminus\{0, 0_\lambda\},
		\]
		we deduce that the support of $w$ is compactly contained in $\mathbb{R}^N_+\cap B_R$ for some $R>0$. Hence, we can use $w\chi_{\hat\Sigma_\lambda}$ as a test function in the weak formulation of \eqref{kelvin} and \eqref{kelvin_r}. By subtracting, we deduce
		\[
		\int_{\hat\Sigma_\lambda} (|\nabla \hat u|^{N-2}\nabla \hat u - |\nabla \hat u_\lambda|^{N-2}\nabla \hat u_\lambda, \nabla w) = \int_{\hat\Sigma_\lambda} \left(\dfrac{1}{|x|^{2N}\hat u^\gamma(x)} - \frac{1}{|x_\lambda|^{2N}\hat u_\lambda^\gamma(x)}\right) w \le 0
		\]
		since $\hat u \ge \hat u_\lambda$ on the support of $w$ and $|x|\ge|x_\lambda|$ in $\hat\Sigma_\lambda$. Using \eqref{p_ge}, we get
		\[
		\int_{\hat\Sigma_\lambda} (|\nabla \hat u| + |\nabla \hat u_\lambda|)^{N-2} |\nabla w| = 0.
		\]
		Hence $w=0$ in $\hat\Sigma_\lambda$, which means $\hat u\le \hat u_\lambda+\varepsilon$. Since $\varepsilon$ is arbitrary, we deduce $\hat u\le \hat u_\lambda$ in $\tilde\Sigma_\lambda$ for all $\lambda<0$.
		
		Repeating the argument in the opposite direction we conclude that $\hat u(x_1, \tilde{x})=\hat u(-x_1, \tilde{x})$. This indicates $u(x_1, \tilde{x})=u(-x_1, \tilde{x})$ for all $x\in\mathbb{R}^N_+$. Since problem \eqref{conformal} is invariant with respect to translation and rotation, we deduce that $u$ is symmetric with respect to any hyperplane perpendicular to $\partial\mathbb{R}^N_+$. In other words, $u$ depends only on $x_N$.	
		The explicit formula for $u$ in the case $\gamma>1$ and a nonexistence result in the case $0<\gamma\le1$ then follows from Theorem \ref{th:1D}.
	\end{proof}
	\section*{Statements and Declarations}
	
	
	\textbf{Data Availability} Data sharing not applicable to this article as no datasets were generated or analysed during the current study.
	
	\textbf{Competing Interests} The author has no competing interests to declare that are relevant to the content of this article.
	
	
	\bibliographystyle{abbrvurl}
	\bibliography{../../../references}
	
\end{document}